\documentclass[reqno]{amsart}

\usepackage{graphicx,subfigure}

\numberwithin{equation}{section}

\usepackage[latin1]{inputenc}
\usepackage[english]{babel}

\usepackage{amsmath,amsthm,amsfonts,latexsym,amssymb}
\usepackage[colorlinks]{hyperref}
\hypersetup{linkcolor=blue,citecolor=blue,filecolor=black,urlcolor=blue}
\usepackage{comment}

\usepackage{color,amsthm,amsfonts}
\definecolor{darkgreen}{rgb}{0,0.7,0.1}

{ \theoremstyle{plain}
\newtheorem{theorem}{Theorem}[section]
\newtheorem{proposition}[theorem]{Proposition}
\newtheorem{lemma}[theorem]{Lemma}

  \theoremstyle{remark}
\newtheorem{remark}[theorem]{Remark}
  \theoremstyle{definition}

}

\def\R{\mathbb{R}}

\def\N{\mathbb{N}}

\def\eps{\varepsilon}

\def\eps{\varepsilon}

\begin{document}


\title[A $p$-Laplacian supercritical Neumann problem]{A $p$-Laplacian supercritical Neumann problem}

\author[F. Colasuonno]{Francesca Colasuonno}
\author[B. Noris]{Benedetta Noris}
\address{D\'epartement de Math\'ematique\newline\indent
Universit\'e Libre de Bruxelles
\newline\indent
Campus de la Plaine - CP214
boulevard du Triomphe - 1050 Bruxelles, Belgique}

\email{francesca.colasuonno@ulb.ac.be}
\email{benedettanoris@gmail.com}

\date{\today}

\begin{abstract}
For $p>2$, we consider the quasilinear equation $-\Delta_p u+|u|^{p-2}u=g(u)$ in the unit ball $B$ of $\mathbb R^N$, with homogeneous Neumann boundary conditions. The assumptions on $g$ are very mild and allow the nonlinearity to be possibly supercritical in the sense of Sobolev embeddings. We prove the existence of a nonconstant, positive, radially nondecreasing solution via variational methods. In the case $g(u)=|u|^{q-2}u$, we detect the asymptotic behavior of these solutions as $q\to \infty$. 
\end{abstract}

\maketitle

\section{Introduction}
For $p>2$, we consider the following Neumann problem
\begin{equation}\label{Pg}
\begin{cases}
-\Delta_p u+u^{p-1}=g(u)\quad&\mbox{in }B,\\
u>0\quad&\mbox{in }B,\\
\partial_\nu u=0\quad&\mbox{on }\partial B.
\end{cases}
\end{equation}
Here $B$ is the unit ball of $\mathbb R^N$, $N\ge1$, and $\nu$ is the outer unit normal of $\partial B$.
We aim to investigate the existence of nonconstant solutions of \eqref{Pg} under very mild assumptions on the nonlinearity $g$, allowing in particular for Sobolev-supercritical growth.

Quasilinear equations with Neumann boundary conditions and subcritical nonlinearities in the sense of Sobolev embeddings have been studied in several papers, among which we refer to \cite{Aizicovici2, Anello, Bonanno, Boscaggi, Faraci,Filippakis,Motreanu,Ricceri, WuTan} and the references therein. 

When $g$ has supercritical growth, a major difficulty in analyzing the existence of solutions of \eqref{Pg} is that, due to the absence of Sobolev embeddings, the energy functional associated to the equation is not well defined in $W^{1,p}(B)$, and so, a priori, it is not possible to apply variational methods. Nonetheless, the problem \eqref{Pg} with the prototype nonlinearity $g(u)=u^{q-1}$ admits the constant solution $u\equiv 1$ for every $q\in(1,\infty)$. This marks a difference with respect to the analogous problem under homogeneous Dirichlet boundary conditions, in which the Poho\v{z}aev identity is an insurmountable obstruction to the existence of non-zero solutions when $q\ge p^*$ (see \cite[Section 2, pp. 685-686]{PS85}). Thus it is a natural question to ask whether \eqref{Pg} also admits nonconstant solutions. 

This question has been tackled in the case $p=2$ and a positive answer has been given in \cite{SerraTilli2011, BonheureSerra2011,BNW,GrossiNoris}. Multiplicity results have been obtained in \cite{BonheureGrossiNorisTerracini2015,BonheureCasterasNoris2016}. The strategy used in \cite{SerraTilli2011, BNW} to obtain existence is that of establishing a priori estimates in some special classes of solutions of \eqref{Pg}. This, in turn, allows to provide a variational characterization of the problem in the Sobolev space. On the other hand, in \cite{BonheureSerra2011, GrossiNoris}, existence is proved by a perturbative method. In \cite{BonheureGrossiNorisTerracini2015, BonheureCasterasNoris2016} the authors apply both a priori estimates and perturbative methods to have multiplicity results. Topological methods have been used in \cite{bonheure2013radial} for a related problem.

The case $p\neq2$ with a supercritical nonlinearity has been treated by S. Secchi in \cite{secchi2012increasing}, where the right-hand side of the equation in \eqref{Pg} is of the type $a(x)g(u)$, with $a(x)$ nonconstant. Our paper aims to extend the results in \cite{secchi2012increasing} to the case $a$ constant and $p>2$. We remark that our method differs from the one in \cite{secchi2012increasing}: whereas S. Secchi adapts to the case $p\neq2$ the techniques introduced in \cite{SerraTilli2011}, we take inspiration from the techniques developed in \cite{BNW}.

We also mention that the existence and multiplicity of solutions to supercritical $p$-Laplacian problems under homogeneous Dirichlet boundary conditions have been studied in several papers, see for instance \cite{Drabek, Clapp16, massa} and the references therein.

In order to state our main result, let us introduce our assumptions on $g$. We assume that $g:[0,\infty)\to\mathbb R$ is of class $C^1([0,\infty))$ and satisfies the following hypotheses

\begin{itemize}
\item[$(g_1)$] $\lim_{s\to 0^+}\frac{g(s)}{s^{p-1}}\in[0,1)$;
\item[$(g_2)$] $\liminf_{s\to\infty}\frac{g(s)}{s^{p-1}}>1$;
\item[$(g_3)$] $\exists$ a constant $u_0>0$ such that $g(u_0)=u_0^{p-1}$ and $g'(u_0)>(p-1)u_0^{p-2}$.
\end{itemize}
\smallskip

We remark that by the regularity of $g$ and by $(g_1)$ and $(g_2)$ we immediately have the existence of an intersection point $u_0>0$ between $g$ and the power function $s^{p-1}$, with $g'(u_0)\ge (s^{p-1})'(u_0)= (p-1)u_0^{p-2}$. Hence, condition $(g_3)$ is only needed to prevent the {\it degenerate} situation in which $g'(u_0)=(p-1)u_0^{p-2}$ at all intersection points $u_0$ such that $g(s)<s^{p-1}$ for $s\in [u_0-\varepsilon,u_0)$ and $g(s)>s^{p-1}$ for $s\in(u_0,u_0+\varepsilon]$.

\begin{theorem}\label{thm:main}
Let $p>2$ and $g \in C^1([0,\infty))$ satisfy assumptions $(g_1)$-$(g_3)$. Then there exists a positive, nonconstant, radial, nondecreasing solution of \eqref{Pg}.

In addition, if $u_{0,1},\ldots,u_{0,n}$ are $n$ different positive constants satisfying $(g_3)$, then \eqref{Pg} admits $n$ different positive, nonconstant, radial, nondecreasing solutions.
\end{theorem}

Our starting point to prove Theorem \ref{thm:main} is to work in the cone of nonnegative, radial, nondecreasing functions  
\begin{equation}\label{cone}
\mathcal C:=\{u\in W^{1,p}_{\mathrm{rad}}(B)\,:\, u\ge0,\,u(r)\le u(s) \mbox{ for all }0<r\le s\le1\},
\end{equation}
introduced by  Serra and Tilli in \cite{SerraTilli2011}, where with abuse of notation we write $u(|x|):=u(x)$. The main advantage of working in this set is the fact that all solutions of \eqref{Pg} belonging to $\mathcal C$ are a priori bounded in $W^{1,p}(B)$ and in $L^\infty(B)$. 
Two strategies are available in literature. The first one, see  \cite{SerraTilli2011, secchi2012increasing}, consists in defining the energy functional $I:\mathcal C\to\mathbb R$  associated to the equation and to find a critical point $u$ of $I$, that is to say 
$$u\in\mathcal C\quad\mbox{such that  }I'(u)[\varphi]=0\quad\mbox{for all } \varphi\in\mathcal C.$$  
This does not imply that $u$ is a weak solution of the problem. Under additional hypotheses on the nonlinearity $g$, the authors prove that it actually is.  

In order to weaken the assumptions on the nonlinearity $g$, we follow a different strategy, see \cite{BNW}. Thanks to the a priori estimates on the solutions of \eqref{Pg} belonging to $\mathcal C$, we are allowed to truncate the nonlinearity $g$. Thus, we deal with a new problem involving a Sobolev-subcritical nonlinearity, with the property that all solutions of the new problem belonging to $\mathcal C$ solve also the original problem \eqref{Pg}. In this way, the energy functional $I$ associated to the truncated problem is well defined in the whole of $W^{1,p}(B)$. To get a solution of \eqref{Pg}, we prove that a mountain pass type theorem holds inside the cone $\mathcal C$. The main difficulty here is the construction of a descending flow that preserves $\mathcal C$.

Once the mountain pass solution is found, we need to prove that it is nonconstant. We further restrict our cone, working in a subset of $\mathcal C$ in which the only constant solution of \eqref{Pg} is the positive constant $u_0$ defined in $(g_3)$. In this set, we build an admissible curve on which the energy is lower than the energy of the constant $u_0$, which gives immediately that the mountain pass solution is not identically equal to $u_0$. 
We remark that this part of the proof heavily relies on the fact that $I$ is of class $C^2$, thus it cannot be generalized to the case $1<p<2$.

In the case in which there is more than one constant $u_0$ satisfying condition $(g_3)$, we work in a restricted cone in order to localize the  mountain pass solution. This allows us to prove the multiplicity result stated in Theorem \ref{thm:main}.

We remark that in the setting $p=2$ our hypoteses are slightly more general than the ones in \cite{BNW}. More precisely, in \cite{BNW} it is required, for $p=2$,
\[
g(s) \textrm{ nondecreasing, } \lim_{s\to 0^+}\frac{g(s)}{s}=0.
\]
Hence our proof also provides the following generalization of  \cite[Theorem 1.3]{BNW}.

\begin{theorem}\label{thm:p=2}
Let $p=2$.  Let $g \in C^1([0,\infty))$ satisfy $(g_1)$, $(g_2)$ and 
\begin{itemize}
\item[$(g_3')$] $\exists$ a constant $u_0>0$ such that $g(u_0)=u_0$ and $g'(u_0)>\lambda_2^{rad}$,
\end{itemize}
where $\lambda_2^{rad}$ is the second radial eigenvalue of $-\Delta+\mathbb I$ in $B$ with Neumann boundary conditions. Then there exists an increasing radial solution of \eqref{Pg}.

In addition, if $u_{0,1},\ldots,u_{0,n}$ are $n$ different positive constants satisfying $(g_3')$, then \eqref{Pg} admits $n$ different increasing radial solutions.
\end{theorem}

For $p=2$ and $g(u)=u^{q-1}$, the result in \cite{BonheureGrossiNorisTerracini2015} provides multiple solutions which oscillate around the constant solution $u\equiv u_0=1$. Similar oscillating solutions can be found via a bifurcation technique, as in \cite{bonheure2016multiple}. There is a branch bifurcating in correspondence to any power $q-1=\lambda_i^{\mathrm{rad}}$ with $i\ge2$, where $\lambda_i^{\mathrm{rad}}$ is the $i$-th eigenvalue of $-\Delta +\mathbb I$ under homogeneous Neumann boundary conditions in the unit ball. We note in passing that this relation between $q$ and $\lambda_i^{\mathrm{rad}}$ seems to be in the same spirit as condition $(g_3')$. It would be interesting to understand whether the bifurcation occurs also when $p>2$. 

Theorem \ref{thm:main} ensures in particular the existence of a nonconstant, nondecreasing, radial solution of \eqref{Pg} in the case $g(u)=u^{q-1}$, for every $q> p$. Denoting by $u_q$ such solution, we detect its asymptotic behavior as $q\to\infty$, in the spirit of \cite{Grossi2006} (see also \cite{GrossiNoris}, \cite{BonheureGrossiNorisTerracini2015}, and \cite{BonheureCasterasNoris2016}).

\begin{theorem}\label{thm:asymptotic_q}
Let $p>2$ and $g(u)=u^{q-1}$, with $q> p$. Denote by $u_q$ the corresponding positive, nonconstant, radially nondecreasing solution found in Theorem~\ref{thm:main}. Then, as $q\to \infty$,
\begin{equation}\label{uqgoestoG}
u_q \to G \textrm{ in } W^{1,p}(B) \cap C^{0,\nu}(\bar B)
\end{equation}
for any $\nu\in(0,1)$, where $G$ is the unique solution of
\begin{equation}\label{eqforG}
\begin{cases}-\Delta_pG+G^{p-1}=0\quad&\mbox{in }B,\\
G=1&\mbox{on }\partial B.
\end{cases}
\end{equation}
\end{theorem}
In the proof of this theorem, we use the fact that the solutions $u_q$ are nondecreasing and that the nonlinearity $g$ is a pure power to get an estimate on the $C^1$-norm of $u_q$, which is uniform in $q$. This ensures the existence of a limit profile function $G$ which is nonnegative and radially nondecreasing. 
We note that it is delicate to prove that $G$ solves the equation in \eqref{eqforG} near the boundary $\partial B$. Heuristically, this comes from the fact that $u_q(1)>1$ for all $q$, and so $\lim_{q\to\infty}u_q(1)^{q-1}$ may be an indeterminate form. In order to prove that $G$ solves actually \eqref{eqforG} in the whole ball $B$, we show that the mountain pass levels $c_q$'s tend to a value $c_\infty$ which is a critical level for the energy associated to \eqref{eqforG}. This latter result requires in turn the preliminary proof of the fact that any mountain pass level $c_q$ coincides with the minimum of the energy functional on a Nehari-type set already introduced in \cite{SerraTilli2011} (that is to say, the Nehari manifold intersected with the cone $\mathcal C$).
We remark here that the Neumann boundary condition is not preserved in the limit, being $\partial_\nu G>0$ on $\partial B$, by Hopf's Lemma (see for instance \cite[Theorem~3.3]{damascelli1999symmetry}). Hence, the convergence $C^{0,\gamma}(\bar B)$ in \eqref{uqgoestoG} is optimal. 

The paper is organized as follows. In Section \ref{sec2}, we prove a priori estimates for nonnegative, radially nondecreasing solutions of \eqref{Pg}. In Section \ref{sec:mountain_pass_existece}, we show the existence of a nonnegative, radially nondecreasing solution of \eqref{Pg} via a mountain pass type argument. Furthermore, in Section \ref{sec4} we conclude the proof of Theorem~\ref{thm:main}, by proving the nonconstancy of the solution found in Section \ref{sec:mountain_pass_existece} and the multiplicity result. A sketch of the proof of Theorem \ref{thm:p=2} is also given in the same section. The asymptotic behavior as $q\to\infty$ of the mountain pass solution of \eqref{Pg} in the pure power case is then studied in Section \ref{sec5}. Finally in Appendix \ref{A}, we collect some partial results valid in the case $1<p<2$.     

\section{A priori bounds for nondecreasing radial solutions}\label{sec2}

\begin{lemma}\label{gtof} For every $g\in C^1([0,\infty))$ satisfying $(g_1)$-$(g_2)$ there exist $f\in C^1([0,\infty))$ nonnegative and nondecreasing, and a constant $m\ge 1$ for which the following properties hold
\begin{itemize}
\item[$(f_1)$] $\lim_{s\to 0^+}\frac{f(s)}{s^{p-1}}\in[m-1,m)$;
\item[$(f_2)$] $\liminf_{s\to\infty}\frac{f(s)}{s^{p-1}}>m$.
\end{itemize}
Furthermore, if $g$ verifies also $(g_3)$, $f$ verifies
\begin{itemize}
\item[$(f_3)$] $\exists$ a constant $u_0>0$ such that $f(u_0)=mu_0^{p-1}$ and $f'(u_0)>m(p-1)u_0^{p-2}$.
\end{itemize}
\end{lemma}
\begin{proof} Since $g\in C^1([0,\infty))$ satisfies $(g_1)$ and $(g_2)$, there exists $C\geq 0$ such that  $$g'(s)\ge -C (p-1) s^{p-2}\quad\mbox{for all }s\in[0,\infty).$$ Hence, if we define $f:[0,\infty)\to\mathbb R$ by $$f(s):= g(s)+Cs^{p-1},$$ $f\in C^{1}([0,\infty))$, $f(0)=0$, $f'\ge0$, and so $f\ge0$. Furthermore, by $(g_1)$, $f$ satisfies 
$$\lim_{s\to 0^+}\frac{f(s)}{s^{p-1}}\in[C,1+C).$$ 
Properties $(f_1)$-$(f_3)$ then follow immediately by $(g_1)$-$(g_3)$, with $m:=1+C$. 
\end{proof}

As a consequence of the previous lemma, from now on in the paper we consider the equivalent problem 
\begin{equation}\label{P}
\begin{cases}
-\Delta_p u+m u^{p-1}=f(u)\quad&\mbox{in }B,\\
u>0\quad&\mbox{in }B,\\
\partial_\nu u=0\quad&\mbox{on }\partial B,
\end{cases}
\end{equation}
where $f\in C^1([0,\infty))$ is nonnegative, nondecreasing, and satisfies $(f_1)$-$(f_3)$. We endow the space $W^{1,p}(B)$ with the equivalent norm $\|\cdot\|:W^{1,p}(B)\to\mathbb R^+$ defined by
$$\|u\|:= \left(\|\nabla u\|_{L^p(B)}^p+m\|u\|_{L^p(B)}^p\right)^{1/p}.$$ 
\smallskip

We look for solutions to \eqref{P} in $W^{1,p}_{\mathrm{rad}}(B)$, that is to say the space of radial functions in $W^{1,p}(B)$. Since $p>1$, we can assume that $W^{1,p}_{\mathrm{rad}}(B)$-functions are continuous in $(0,1]$ and define the {\it cone of nonnegative radially nondecreasing functions} as in \eqref{cone}.
We note that, if $u\in\mathcal C$, we can set $u(0):=\lim_{r\to0^+}u(r)$ by monotonicity, and consider $u\in C(\bar{B})$. Moreover, being nondecreasing, every $u\in\mathcal C$ is differentiable a.e. and $u'(r)\ge0$ where it is defined. It is easy to prove that $\mathcal C$ is a closed convex cone in $W^{1,p}(B)$, that is to say, the following properties hold for all $u,\,v\in\mathcal C$ and $\lambda\ge0$
\begin{itemize}
\item[(i)] $\lambda u\in \mathcal C$;
\item[(ii)] $u+v\in \mathcal C$;
\item[(iii)] if also $-u\in\mathcal C$, then $u\equiv0$;
\item[(iv)] $\mathcal C$ is closed for the topology of $W^{1,p}$.
\end{itemize}

The cone $\mathcal{C}$ was first introduced in \cite{SerraTilli2011} in the case $p=2$. It is a useful set when working with Sobolev-supercritical problems because of the following a priori estimates.

\begin{lemma}\label{bounded} For every $1\leq q< \infty$ there exists $C(N,q)$ such that 
$$\|u\|_{L^\infty(B)}\le C(N,q)\|u\|_{W^{1,q}(B)}\quad\mbox{for all }u\in\mathcal{C}.$$
\end{lemma}
\begin{proof} Since $u\in\mathcal C$ is nonnegative and nondecreasing, we get
\begin{equation}\label{eq:infinity_norm_increasing}
\|u\|_{L^\infty(B)}=\|u\|_{L^\infty(B\setminus B_{1/2})},
\end{equation}
and by the radial symmetry of $u\in\mathcal C$, 
\begin{equation}\label{eq:a_priori_radial}
\|u\|_{L^\infty(B\setminus B_{1/2})}\le C\|u\|_{W^{1,1}(B\setminus B_{1/2})}\le C\|u\|_{W^{1,1}(B)}
\end{equation}
for some $C>0$ depending only on the dimension $N$. Moreover, being $B$ bounded, for every $q\in[1,\infty)$ there exists a constant $C>0$ depending only on $N$ and $q$, such that 
\begin{equation}\label{eq:a_priori_B}
\|u\|_{W^{1,1}(B)}\le C\|u\|_{W^{1,q}(B)}\quad\mbox{for all }u\in W^{1,1}(B).
\end{equation} 
By combining \eqref{eq:infinity_norm_increasing}, \eqref{eq:a_priori_radial} and \eqref{eq:a_priori_B}, for every $q\in[1,\infty)$ we can find a constant $C(N,q)>0$ for which the statement holds.
\end{proof}

\begin{lemma} For all $q\in[1,\infty)$, the cone $\mathcal C$ endowed with the $W^{1,p}$-norm is compactly embedded in $L^q(B)$.
\end{lemma}
\begin{proof} If $N<p$ the conclusion follows at once by the Rellich-Kondrachov theorem. In the complementary case, we take into account the fact that $\mathcal C$-functions are bounded. More precisely, if we have $(u_n)\subset\mathcal C$ bounded in the $W^{1,p}$-norm, there exists $u\in\mathcal C$ such that up to a subsequence $u_n\rightharpoonup u$ in $W^{1,p}(B)$ and so $u_n\to u$ in $L^1(B)$. Therefore, by Lemma~\ref{bounded} we get that for every $q<\infty$, 
$$\begin{aligned}\int_B|u_n-u|^qdx &\le\|u_n-u\|_{L^\infty(B)}^{q-1}\|u_n-u\|_{L^1(B)}\\
&\le C(N,p)^{q-1}\|u_n-u\|_{W^{1,p}(B)}^{q-1}\|u_n-u\|_{L^1(B)}
\to0,\end{aligned}$$  
that is $u_n\to u$ in $L^q(B)$.
\end{proof}

Fix $\delta,\, M>0$ such that 
 \begin{equation}\label{Mdelta}
 f(s)\ge(m+\delta)s^{p-1}\quad\mbox{for all }s\ge M.
 \end{equation}
The existence of $\delta,\, M>0$ follows by $(f_2)$ in Lemma \ref{gtof}.
We introduce the following set of functions  
\begin{equation}\label{FmdeltaM}
\mathfrak{F}:=\left\{\varphi\in C([0,\infty))\,:\,\varphi\mbox{ nonnegative, } \, \varphi(s)\ge(m+\delta)s^{p-1}\,\mbox{for all }s\ge M\right\}
\end{equation}
We remark that $\mathfrak F$ depends on $f$ only through $\delta$ and $M$. In the remaining of this section, we shall derive some a priori estimates which are uniform in $\mathfrak F$ and hence depend only on $\delta$ and $M$ and not on the specific nonlinearity $f$ belonging to $\mathfrak F$.

\begin{lemma} There exists a constant $K_{p-1}>0$ such that  
$$\|u\|_{L^{p-1}(B)}\le K_{p-1}$$
for every solution $u$ of 
\begin{equation}\label{Pphi}
\begin{cases}
-\Delta_p u+m u^{p-1}=\varphi(u)\quad&\mbox{in }B,\\
u>0&\mbox{in }B,\\
\partial_\nu u=0&\mbox{on }\partial B,
\end{cases}
\end{equation}
and for every $\varphi\in\mathfrak F$.
\end{lemma}
\begin{proof} By integrating the equation in \eqref{Pphi} and using the fact that $\varphi\in\mathfrak F$, we have
$$m\int_B u^{p-1}dx=\int_{\{u< M\}} \varphi(u)dx+\int_{\{u\ge M\}}\varphi(u)dx\ge(m+\delta)\int_{\{u\ge M\}}u^{p-1}dx.$$
Thus,
$$mM^{p-1}|B|>m\int_{\{u< M\}}u^{p-1}dx\ge\delta\int_{\{u\ge M\}}u^{p-1}dx,$$
where $|B|$ is the volume of the unitary ball of $\mathbb R^N$, and so 
\begin{equation}\label{p-1}\int_B u^{p-1}dx=\int_{\{u< M\}}u^{p-1}dx+\int_{\{u\ge M\}}u^{p-1}dx<\left(1+\frac{m}\delta\right)M^{p-1}|B|=:K_{p-1}^{p-1},
\end{equation}
which yields the estimate. 
\end{proof}

\begin{lemma}\label{apriori} There exists a constant $K_{\infty}>0$ such that  
$$\|u\|_{L^\infty(B)}\le K_\infty \quad\mbox{and}\quad \|u\|\le \left(K_\infty |B|\max_{s\in[0,K_\infty]}\varphi(s)\right)^{1/p}$$
for every solution $u\in\mathcal C$ of \eqref{Pphi} and every $\varphi\in\mathfrak F$.
\end{lemma}
\begin{proof}
Let $u\in\mathcal C$ be a solution of \eqref{Pphi}. We recall that the $p$-Laplacian of a radial function
is given by
$$\Delta_p u=\frac1{r^{N-1}}\left(r^{N-1}|u'(r)|^{p-2}u'(r)\right)'=|u'(r)|^{p-2}\left[(p-1)u''(r)+\frac1r(N-1)u'(r)\right].$$
Hence, since $u'\ge0$ a.e., we can write
$$\begin{cases}(r^{N-1}u'(r)^{p-1})'=r^{N-1}(mu^{p-1}-\varphi(u))\quad\mbox{in }(0,1),\\
u'(0)=u'(1)=0.\end{cases}$$
Then, by integrating the equation over the interval $(0,r)$ and using the fact that $\varphi$ is nonnegative, we have
$$\begin{aligned}r^{N-1}u'(r)^{p-1}&=\int_0^r (m\,u(t)^{p-1}-\varphi(u(t)))t^{N-1} dt\\
&\le m\int_0^r u(t)^{p-1}t^{N-1}dt=\frac{m}{|\partial B|}\int_B u(x)^{p-1}dx,\end{aligned}$$
where $|\partial B|$ is the $(N-1)$-dimensional measure of the unitary sphere in $\mathbb R^N$. Together with \eqref{p-1}, this gives
$$\|u\|_{W^{1,p-1}(B)}\le (1+m)^{1/(p-1)}K_{p-1}.$$
The first estimate then follows by Lemma~\ref{bounded} (by taking $q=p-1$), with $K_\infty:=(1+m)^{1/(p-1)}K_{p-1}C(N,p-1)$.
Finally, for the last estimate, we multiply the equation of \eqref{Pphi} by $u$, we integrate over $B$, and we obtain
$$\|u\|^p=\int_B\varphi(u)udx\le K_\infty|B|\max_{s\in[0,K_\infty]}\varphi(s),$$
which concludes the proof.
\end{proof}

\section{Existence of a mountain pass radial solution}\label{sec:mountain_pass_existece}

In this section we prove the existence of a radial solution of \eqref{P} via the Mountain Pass Theorem. Since the nonlinearity $f$ is possibly supercritical in the sense of Sobolev spaces, we need to truncate and to replace it by a subcritical function which coincides with $f$ in $[0,K_\infty]$, $K_\infty$ being defined in Lemma~\ref{apriori}. Then, we take advantage of the a priori estimates proved in the previous section to guarantee that the mountain pass solution found with the truncated function is indeed a solution of the original problem \eqref{P}.\medskip     

We define the critical Sobolev exponent $$p^*:=\begin{cases}\frac{Np}{N-p}\quad&\mbox{if }p<N,\\
+\infty&\mbox{otherwise}.\end{cases}$$

\begin{lemma}\label{truncated} For every $\ell\in(p,p^*)$, there exists $\tilde{f}\in \mathfrak F\cap C^1([0,\infty))$ nondecreasing, satisfying $(f_1)$-$(f_3)$, 
 \begin{equation}\label{subcritical}
\lim_{s\to\infty}\frac{\tilde{f}(s)}{s^{\ell-1}}=1 ,
\end{equation}
and with the property that if $u\in\mathcal C$ solves 
\begin{equation}\label{tildeP}\begin{cases}-\Delta_p u+mu^{p-1}=\tilde{f}(u)\quad&\mbox{in }B,\\
u>0&\mbox{in }B,\\
\partial_\nu u=0&\mbox{on }\partial B,
\end{cases}
\end{equation}
then $u$ solves \eqref{P}.  
\end{lemma}
\begin{proof} Let $\delta>0$ and $M>0$ be the constants defined in \eqref{Mdelta}, and fix $s_0>\max\{K_\infty,M\}$, with $K_\infty$ given in Lemma \ref{apriori}. By \eqref{Mdelta}, two possible cases arise.\smallskip

{\it Case $f(s_0)= (m+\delta)s_0^{p-1}$.} By \eqref{Mdelta}, $f$ is tangent at $s_0$ to the curve $(m+\delta)s^{p-1}$, hence $f'(s_0)=(m+\delta)(p-1)s_0^{p-2}$ and we can define the function $\tilde f:[0,\infty)\to[0,\infty)$ as
$$\tilde{f}(s):=\begin{cases}f(s)\quad&\mbox{if }s\in[0,s_0],\\
f(s_0)+(m+\delta)(s^{p-1}-s_0^{p-1})+(s-s_0)^{\ell-1} &\mbox{otherwise}.\end{cases}$$ 
\smallskip  

{\it Case $f(s_0)> (m+\delta)s_0^{p-1}$.}  First, we modify $f$ in a right neighborhood of $s_0$, that is to say,  we consider a $C^1$ nondecreasing function $f_{\mathrm{mod}}:[0,s_0+\varepsilon]\to[0,\infty)$ in such a way that $f_{\mathrm{mod}}(s)=f(s)$ in $[0,s_0]$, 
$f_{\mathrm{mod}}(s)\ge (m+\delta)s^{p-1}$ in $[s_0,s_0+\varepsilon]$, and $f'_{\mathrm{mod}}(s_0+\varepsilon)=(m+\delta)(p-1)(s_0+\varepsilon)^{p-2}$. Then, we define $\tilde{f}(s)$ as in the previous case, with $f$ replaced by $f_{\mathrm{mod}}$ and $s_0$ by $s_0+\varepsilon$.  
\smallskip

In both cases, it is easy to check that $\tilde{f}$ is of class $C^1([0,\infty))$, nonnegative, nondecreasing, satisfies $(f_1)$, $(f_2)$, and \eqref{subcritical}. Since the constant $u_0$ given in $(f_3)$ is a solution of \eqref{P} in $\mathcal{C}$, we know by Lemma~\ref{apriori} that $u_0\le K_\infty< s_0$. Hence $\tilde f$ verifies also $(f_3)$. 

Finally, let $u\in\mathcal C$ solve \eqref{tildeP}, we want to show that $u$ solves \eqref{P}. To this aim, we notice that $\tilde f$ belongs to $\mathfrak F$ by construction. By Lemma \ref{apriori}, $\|u\|_{L^\infty(B)}<K_\infty$. Being $s_0>K_\infty$, we have $\tilde f(u)=f(u)$, hence $u$ solves \eqref{P}.
\end{proof}

As a consequence of the proof of the previous lemma, there exists $C>0$ for which 
\begin{equation}\label{crescitasottocritica}
\tilde f(s)\le C (1+s^{\ell-1})\quad\mbox{for all }s\ge0.
\end{equation}
From now on in the paper, we set $\tilde f= 0$ in $(-\infty, 0)$. We define the energy functional $I:W^{1,p}(B)\to\mathbb R$ associated to the problem \eqref{tildeP} by
\begin{equation}\label{I}
I(u):= \int_B\left(\frac{|\nabla u|^p+m|u|^p}p-\tilde{F}(u)\right)dx,
\end{equation}  
where $\tilde F(u):=\int_0^u \tilde f(s)ds$. Because of \eqref{subcritical} and the Sobolev embedding, the functional $I$ is well defined and of class $C^2$, being $p>2$.  

We define the operator $T:(W^{1,p}(B))'\to W^{1,p}(B)$ as 
\begin{equation}\label{T}T(w)=v, \quad\mbox{where $v$ solves }
\quad(P_w)\;\begin{cases}-\Delta_p v+m|v|^{p-2}v=w\quad&\mbox{in }B,\\
\partial_\nu v=0&\mbox{on }\partial B.
\end{cases}
\end{equation}

We observe that the definition is well posed because, for all $w\in (W^{1,p}(B))'$, the problem $(P_w)$ admits a unique weak solution $v\in W^{1,p}(B)$. To prove the existence one can apply the direct method of Calculus of Variations, while uniqueness is a consequence of the strict convexity of the map $u\mapsto \|u\|^p$. Furthermore, by \cite[Lemma 2.1]{BonderRossi} we know that 
\begin{equation}\label{eq:T_continuous}
T\in C((W^{1,p}(B))';W^{1,p}(B)).
\end{equation}

We introduce also the operator
\begin{equation}\label{eq:tildeT_def}
\tilde T:W^{1,p}(B)\to W^{1,p}(B) \quad\textrm{defined by}\quad \tilde T(u)=T(\tilde f(u)),
\end{equation}
with $T$ given in \eqref{T}. Being $\ell<p^*$, $u\in W^{1,p}(B)$ implies $u\in L^\ell(B)$. Hence, by \eqref{crescitasottocritica}, $\tilde f(u)\in L^{\ell'}(B)\subset (W^{1,p}(B))'$ and $\tilde T$ is well defined.

\begin{proposition}\label{Ttildecompact}
The operator $\tilde T$ is compact, i.e. it maps bounded subsets of $W^{1,p}(B)$ into precompact subsets of $W^{1,p}(B)$. Furthermore, there exist two positive constants $a,\,b$ such that for all $u\in W^{1,p}(B)$ the following properties hold
\begin{equation}\label{J2}\begin{aligned}
&I'(u)[u-\tilde T(u)]\ge a\|u-\tilde T(u)\|^p,\\
&\|I'(u)\|_{*}\le b\|u-\tilde T(u)\|(\|u\|+\|\tilde T(u)\|)^{p-2}.
\end{aligned}
\end{equation}
\end{proposition}
\begin{proof} Let $(u_n)$ be a bounded sequence in the reflexive Banach space $W^{1,p}(B)$. Up to a subsequence $u_n\rightharpoonup u$ in $W^{1,p}(B)$ and $u_n\to u$ in $L^\ell(B)$, being $W^{1,p}(B)$ compactly embedded in $L^\ell(B)$.

Now, we claim that $\tilde f(u_n)\to\tilde f(u)$ in $(W^{1,p}(B))'$. Once the claim is proved, the first part of the statement follows by using the continuity \eqref{eq:T_continuous} of the operator $T$. We pick any subsequence, still denoted by $(u_n)$, and we know that, up to another subsequence, $u_n\to u$  a.e. in $B$ and that there exists $h\in L^\ell(B)$ such that $|u_n|\le h$ a.e. in $B$ for all $n$. By the continuity of $\tilde f$ we get that $|\tilde f(u_n)-\tilde f(u)|^{\ell'}\to0$ a.e. in $B$ and that $|\tilde f(u_n)-\tilde f(u)|^{\ell'}\le 2^{\ell'-1}(\tilde f(u_n)^{\ell'}+\tilde f(u)^{\ell'})\le C(1+h^\ell)\in L^1(B)$. Hence, the Dominated Convergence Theorem guarantees that $\tilde f(u_n)\to\tilde f(u)$ in $L^{\ell'}(B)$. By the arbitrariness of the subsequence picked, we have that the same convergence result holds for the whole sequence $(u_n)$. The claim follows at once from the embedding $L^{\ell'}(B)\hookrightarrow (W^{1,p}(B))'$. 

Finally, inequalities \eqref{J2} follow by \eqref{crescitasottocritica} as in the proof of~\cite[Lemmas~3.7, 3.8]{BL}.
\end{proof}

\begin{remark}\label{TfixI'} 
We observe here that \eqref{J2} implies that $\{u\,:\,\tilde T(u)=u\}$ coincides with the set of critical points of $I$.
\end{remark}

\begin{lemma}[\textbf{Palais-Smale condition}]\label{PalaisSmale} The functional $I$ satisfies the Palais-Smale condition, i.e. every sequence $(u_n)\subset W^{1,p}(B)$ such that 
$$(I(u_n)) \mbox{ is bounded}\quad\mbox{ and }\quad I'(u_n)\to 0\mbox{ in }(W^{1,p}(B))'$$ 
admits a convergent subsequence.
\end{lemma}
\begin{proof} Let $(u_n)\subset W^{1,p}(B)$ be a (PS)-sequence for $I$ as in the statement. By \eqref{subcritical} and L'H\^opital's rule, we get 
$$\lim_{s\to+\infty}\frac{\tilde{F}(s)}{s^\ell}=\lim_{s\to+\infty}\frac{\tilde f(s)}{\ell s^{\ell-1}}=\frac1\ell.$$
Thus, $$\lim_{s\to+\infty}\frac{\tilde{f}(s)s}{\tilde{F}(s)}=\lim_{s\to+\infty}\frac{\tilde{f}(s)}{s^{\ell-1}}\frac{s^\ell}{\tilde{F}(s)}=\ell$$
and so, there exist $\mu\in(p,\ell]$ and $R_0>0$ such that $\tilde{f}(s)s\ge\mu\tilde{F}(s)$ for all $s\ge R_0$. Now, we estimate
$$I(u_n)-\frac1\mu I'(u_n)[u_n]\ge \left(\frac1p-\frac1\mu\right)\|u_n\|^p+\int_{\{u_n\le R_0\}}\left(\frac1\mu\tilde{f}(u_n)u_n-\tilde F(u_n)\right)dx$$
and, being $(u_n)$ a (PS)-sequence,  
$$I(u_n)-\frac1\mu I'(u_n)[u_n]\le |I(u_n)|+\frac1\mu \|I'(u_n)\|_*\|u_n\|\le C(1+\|u_n\|),$$
for some $C>0$, where we have denoted by $\|\cdot\|_*$ the norm of the dual space of $W^{1,p}(B)$. Since we know that $\int_{\{u_n\le R_0\}}\left(\frac1\mu\tilde{f}(u_n)u_n-\tilde F(u_n)\right)dx$ is uniformly bounded in $n$, we get 
$$\left(\frac1p-\frac1\mu\right)\|u_n\|^p\le C(1+\|u_n\|).$$ Therefore, $(u_n)$ is bounded in $W^{1,p}(B)$ and there exists $u\in W^{1,p}(B)$ such that $u_n\rightharpoonup u$ in $W^{1,p}(B)$. Hence, Proposition \ref{Ttildecompact} guarantees that, up to a subsequence, $\tilde{T}(u_n)\to\tilde T(u)$ in $W^{1,p}(B)$. This implies, by the triangle inequality, that 
\begin{equation}\label{trineq}
\limsup_{n\to\infty}\|u_n-\tilde T(u)\|\le \limsup_{n\to\infty}\|u_n-\tilde T(u_n)\|.
\end{equation}
On the other hand, by the first inequality of \eqref{J2} we obtain 
$$\|u_n-\tilde T(u_n)\|^p\le \frac{C}a\|I'(u_n)\|_*\to 0.$$
Together with \eqref{trineq}, we conclude that $u_n\to \tilde T(u)=u$ in $W^{1,p}(B)$.
\end{proof}

We define 
\begin{equation}\label{u-u+}
\begin{aligned}
u-&:= \sup \{t \in [0,u_0)\,:\, \tilde f(t)=mt^{p-1}\},\\
u_+&:= \inf \{t \in (u_0,+\infty) \,:\, \tilde f(t)=mt^{p-1}\}.
\end{aligned}
\end{equation}
By Lemma \ref{truncated}, $\tilde f$ satisfies $(f_3)$, so that $u_0$ is an isolated 
zero of the function $\tilde f(t)-mt^{p-1}$, hence 
\begin{equation}\label{u-+}
u_-\neq u_0\quad\mbox{ and }\quad u_+\neq u_0.
\end{equation} 
We point out that $u_+= +\infty$ is possible. Next, we define the set  
\begin{equation}\label{Cstar}
\mathcal C_*:= \{u \in \mathcal C\::\: \text{$u_- \le  u \le u_+$ in $B$}\}.
\end{equation}
Clearly, $\mathcal C_*$ is closed and convex. 

\begin{lemma}\label{cononelcono}
The operator $\tilde T$ defined in \eqref{eq:tildeT_def} satisfies
$\tilde T(\mathcal C_*)\subseteq \mathcal C_*$.
\end{lemma}
\begin{proof} We first note that $u\in\mathcal C_*$ implies $\tilde f(u)\in\mathcal C$, by the properties of $\tilde f$. Now, let $u\in\mathcal C_*$ and $v:=\tilde T(u)$. By standard regularity theory (see e.g. \cite[Theorem 2]{Lieberman}), $v\in C^{1,\alpha}(\bar B)$ for some $\alpha\in(0,1)$. Therefore, by \cite[Theorem 1.1]{Damascelli}, we know that $v\ge0$ in $B$. Furthermore, due to uniqueness, $v$ is radial. Now we prove that $v$ is nondecreasing. It is enough to show that for every $r\in(0,1)$ one of the following cases occurs:
\begin{itemize}
\item[$(a)$] $v(s)\le v(r)$ for all $s\in(0,r)$, 
\item[$(b)$] $v(s)\ge v(r)$ for all $s\in(r,1)$. 
\end{itemize}
Indeed, if $v(t)>v(r)$ for some $t<r$, by the continuity of $v$, there exists $s\in(t,r)$ for which $v(t)>v(s)> v(r)$ which violates both $(a)$ and $(b)$.   
Now, we fix $r\in(0,1)$. If $\tilde f(u(r))\le m v(r)^{p-1}$, we consider the test function
$$\varphi(x):=\begin{cases}(v(|x|)-v(r))^+\quad&\mbox{if }|x|\le r,\\
0&\mbox{otherwise} 
\end{cases}$$
and we have
$$\begin{aligned}\int_{B_r}(|\nabla v|^{p-2}\nabla v\cdot\nabla \varphi+m v^{p-1}\varphi)dx&=\int_{B_r}\tilde f(u)\varphi dx\le\tilde{f}(u(r))\int_{B_r}\varphi dx\\&\le m v(r)^{p-1}\int_{B_r}\varphi dx.\end{aligned}$$
Hence, 
$$\int_{B_r}|\nabla \varphi|^{p}dx+m\int_{B_r}(v(|x|)^{p-1}-v(r)^{p-1})(v(|x|)-v(r))^+ dx\le0,$$
that is $\varphi\equiv 0$, i.e., the case $(a)$ occurs. Analogously, if $\tilde f(u(r))> m v(r)^{p-1}$, we consider the test function
$$\varphi(x):=\begin{cases}0\quad&\mbox{if }|x|\le r,\\
(v(|x|)-v(r))^-&\mbox{otherwise} 
\end{cases}$$ 
and we prove that $(b)$ holds. Therefore, we have proved that $v$ is nondecreasing. 

It remains to show that $u_-\le v\le u_+$. By the fact that $\tilde{f}(u_-)=mu_-^{p-1}$ and that $\tilde f$ is nondecreasing we get 
$$-\Delta_p (v-u_-)+m(v^{p-1}-u_-^{p-1})=\tilde f(u)-\tilde f(u_-)\ge0.$$
Hence, if we multiply the equation above by $(v-u_-)^-$ and integrate it over $B$, we obtain
$$-\|\nabla (v-u_-)^-\|_p^p-m\int_{B}(u_-^{p-1}-v^{p-1})(v-u_-)^-dx\ge 0, $$
that is $(v-u_-)^-\equiv 0$ in $B$. Similarly, if $u_+<+\infty$ we prove that $v\le u^+$ in $B$. 
\end{proof}

\begin{lemma}[\textbf{Locally Lipschitz vector field}]\label{KTI} Let $W:=W^{1,p}(B)\setminus\{u\,:\,\tilde T(u)=u\}$. There exists a locally Lipschitz continuous operator $K: W\to W^{1,p}(B)$ satisfying the following properties: 
\begin{itemize}
\item[(i)] $K(\mathcal C_*\cap W)\subset \mathcal C_*$;
\item[(ii)] $\frac12\|u-K(u)\|\le\|u-\tilde T(u)\|\le2\|u-K(u)\|$ for all $u\in W$;
\item[(iii)] for all $u\in W$
$$I'(u)[u-K(u)]\ge\frac{a}2 \|u-\tilde T(u)\|^p,$$
where $a>0$ is the constant given in Proposition \ref{Ttildecompact}.
\end{itemize}
\end{lemma}
\begin{proof} We follow the arguments in the proofs of \cite[Lemma 4.1]{BL} and \cite[Lemma~2.1]{BartschLiuWeth}. We define the continuous functions $\delta_1,\,\delta_2: W\to \mathbb R$ as
$$
\delta_1(u):=\frac12\|u-\tilde T(u)\|\quad\mbox{and}\quad
\delta_2(u):=\frac{a\|u-\tilde T(u)\|^{p-1}}{2b(\|u\|+\|\tilde T(u)\|)^{p-2}},
$$
where $a,\,b$ are the constants introduced in Proposition \ref{Ttildecompact}.
First we claim that for every $u\in W$, we can find a radius $\varrho(u)>0$ such that for every $v,\,w\in N(u):=\{\phi\in W^{1,p}(B)\,:\,\|\phi-u\|<\varrho(u)\}$ it results that
\begin{equation}\label{mind1d2}
\|\tilde T(v)-\tilde T(w)\|<\min\{\delta_1(v),\,\delta_2(v),\,\delta_1(w),\,\delta_2(w)\}.
\end{equation}
Indeed, suppose by contradiction that for every $n\in\mathbb N$ we can find $v_n,\,w_n\in N_n(u):=\{\phi\in W^{1,p}(B)\,:\,\|\phi-u\|<\frac1n\}$ for which  
\begin{equation}\label{contradictiontt}
\|\tilde T(v_n)-\tilde T(w_n)\|\ge\delta_1(v_n)=\frac12\|v_n-\tilde T(v_n)\|.
\end{equation}
Since $v_n,\,w_n\in N_n(u)$ for every $n$, and by the continuy of $\tilde T$, we get 
$$\lim_{n\to\infty}\|v_n-u\|=\lim_{n\to\infty}\|w_n-u\|=0\quad\mbox{and}\quad\lim_{n\to\infty}\|\tilde T(v_n)-\tilde T(w_n)\|=0.$$
Hence, passing to the limit in \eqref{contradictiontt}, we obtain by the second inequality in \eqref{J2}
$$0\ge\frac12\|u-\tilde T(u)\|\ge\frac{\|I'(u)\|_*}{2b(\|u\|+\|\tilde T(u)\|)^{p-2}}>0,$$
where we have used Remark \ref{TfixI'} and the fact that $u\in W$ implies that $u\not\equiv 0$. This yields a contradiction. Proceeding analogously, one proves that the claim holds. 

Now, let $\mathcal U$ be a locally finite open refinement of $\{N(u)\,:\,u\in W\}$ and let $\{\pi_U\,:\,U\in\mathcal U\}$ be the standard partition of unity subordinated to $\mathcal U$, i.e. 
$$\pi_U(u):=\frac{\alpha_U(u)}{\sum_{V\in\mathcal U}\alpha_V(u)},\qquad\alpha_U(u):=\mathrm{dist}(u,W\setminus U).$$
Clearly, $\sum_{U\in\mathcal U}\pi_U(u)=1$ for any $u\in W$, $\pi_U$ is Lipschitz continuous, it satisfies $\mathrm{supp}(\pi_U)\subseteq U$ and $0\le\pi_U\le1$ for any $U\in\mathcal U$. Furthermore, since $\mathcal U$ is a refinement of $\{N(u)\,:\,u\in W\}$, given any $U\in\mathcal U$, \eqref{mind1d2} holds in particular for any $v,\,w\in U$.
 
For every $U\in\mathcal U$ we choose an element $a_U\in U$ with the property that, if $U\cap\mathcal C_*\neq\emptyset$, then $a_U\in U\cap \mathcal C_*$. We define $K:W\to W^{1,p}(B)$ as
$$K(u):=\sum_{U\in\mathcal U}\pi_U(u)\tilde T(a_U).$$
Therefore, $K$ is locally Lipschitz continuous due to the Lipschitz continuity of any $\pi_U$ and to the local finiteness of the refinement $\mathcal U$. 

Moreover, (i) holds thanks to the facts that $\tilde T$ preserves the cone $\mathcal C_*$ (see Lemma~\ref{cononelcono}), that $K$ is a convex combination of points $\tilde T(a_U)$, and that $\mathcal C_*$ is convex. 

In view of the proof of (ii), by using the properties of the functions $\pi_U$ and \eqref{mind1d2}, we estimate
\begin{equation}\label{K-Ttilde}
\|K(u)-\tilde T(u)\|\le \sum_{\underset{u\in U}{U\in \mathcal U}}\pi_U(u)\|\tilde T(a_U)-\tilde T(u)\|<d_1(u)=\frac12\|u-\tilde T(u)\|.
\end{equation}
This gives immediately  
$$\begin{aligned}
\|u-K(u)\|&\le \|K(u)-\tilde T(u)\|+\|u-\tilde T(u)\|< \frac32\|u-\tilde T(u)\|,\\
\|u-\tilde T(u)\|&\le \|u-K(u)\|+\|K(u)-\tilde T(u)\|\le \|u-K(u)\|+\frac12\|u-\tilde T(u)\|,
\end{aligned}$$
which imply the two inequalities of (ii).

By using the definition of $\delta_2$ and \eqref{J2}, we can finally prove (iii). Indeed, for every $u\in W$
$$
\begin{aligned}
I'(u)[u-K(u)]& \ge I'(u)[u-\tilde T(u)]- \|I'(u)\|_*\|K(u)-\tilde T(u)\|\\
&>a\|u-\tilde T(u)\|^p-b\|u-\tilde T(u)\|(\|u\|+\|\tilde T(u)\|)^{p-2}\delta_2(u)\\
&=\frac{a}2\|u-\tilde T(u)\|^p,
\end{aligned}
$$
and the proof is concluded.
\end{proof}

Without loss of generality, we will take from now on 
\begin{equation}\label{defell}
\ell\in\left(p,\min\left\{\frac{(p-1)^2+p-2}{p-2},p^*\right\}\right),
\end{equation}
where $\ell$ is the subcritical growth of $\tilde f$ defined in Lemma \ref{truncated}. This further condition is needed in the following two lemmas.

\begin{lemma}\label{le:salvezza} For all $c\in\mathbb R$ there exists $C=C(c)>0$ for which 
\begin{equation}\label{salvezza}
\|u\|+\|\tilde T(u)\|\le C\left(1+\|u-\tilde T(u)\|^{\beta}\right),\quad\beta:=\frac{\ell-1}{(p-1)^2-(p-2)(\ell-1)}
\end{equation}
holds for every $u\in W^{1,p}(B)$ with $I(u)\le c$.
\end{lemma}
\begin{proof} 
Reasoning as in the first part of the proof of Lemma~\ref{PalaisSmale}, we easily get, by the fact that $I(u)\le c$, that  
$$\|u\|^p\le C(1+\|I'(u)\|_*\|u\|),$$
for some positive constant $C$.
Then, by the second inequality of \eqref{J2} and by Young's inequality with exponents $(p',p)$, 
$$\begin{aligned}\|u\|^p&\le C(1+\|u-\tilde T(u)\|(\|u\|+\|\tilde T(u)\|)^{p-2}\|u\|)\\
&\le C(1+\|u-\tilde T(u)\|^{p'}(\|u\|+\|\tilde T(u)\|)^{p-p'}+\varepsilon\|u\|^p),\end{aligned}$$
where $C>0$ may change from line to line, and $\varepsilon>0$ is sufficiently small.
Hence, 
\begin{equation}\label{stimautile}
\|u\|\le C(1+\|u-\tilde T(u)\|^{p'/p}(\|u\|+\|\tilde T(u)\|)^{1-p'/p}).
\end{equation}
Young's inequality with exponents $(p/p',(p-1)/(p-2))$ then gives
\begin{equation}\label{per u}\|u\|\le C[1+\|u-\tilde T(u)\|+\varepsilon(\|u\|+\|\tilde T(u)\|)]\quad\mbox{for some }C>0, \varepsilon>0\mbox{ small.}
\end{equation}

Now, consider the equation satisfied by $v:=\tilde T(u)$ 
$$-\Delta_p v+m|v|^{p-2}v=\tilde f(u)\quad\mbox{in }B.$$ 
By testing it with $v$ and using \eqref{crescitasottocritica}, H\"older's inequality, and the Sobolev embedding, we get
$$\begin{aligned}
\|v\|^p&=\int_B \tilde f(u) vdx\le \int_B C(1+u^{\ell-1})vdx\le C\left(\int_B(1+u^{\ell})dx\right)^{1/\ell'}\left(\int_B v^\ell dx\right)^{1/\ell}\\
&\le C \|v\|_{L^\ell(B)}\left(|B|^{1/\ell'}+\|u\|^{\ell/\ell'}_{L^\ell(B)}\right)\le  C \|v\|(1+\|u\|^{\ell/\ell'}),
\end{aligned}$$  
that is $\|\tilde T(u)\|\le C(1+\|u\|^{(\ell-1)/(p-1)})$. By \eqref{stimautile}, we obtain
$$\|\tilde T(u)\|\le C\left(1+\|u-\tilde T(u)\|^{\frac{\ell-1}{(p-1)^2}}(\|u\|+\|\tilde T(u)\|)^{\frac{(\ell-1)(p-2)}{(p-1)^2}}\right).$$
By applying Young's inequality with exponents $\big(\frac{\ell-1}{(p-1)^2-(p-2)(\ell-1)}, \, \frac{(p-1)^2}{(\ell-1)(p-2)}\big)$, we have
$$\|\tilde T(u)\|\le C\left[1+\|u-\tilde T(u)\|^{\frac{\ell-1}{(p-1)^2-(p-2)(\ell-1)}}+\varepsilon(\|u\|+\|\tilde T(u)\|)\right]$$
for some $C>0$, $\varepsilon>0$ small. Together with \eqref{per u}, this implies the thesis.
\end{proof}

\begin{lemma}\label{conseqPS} Let $c\in\mathbb R$ be such that $I'(u)\neq 0$ for all $u\in \mathcal C_*$ with $I(u)=c$. Then, there exist two positive constants $\bar\varepsilon$ and $\bar\delta$ such that the following inequalities hold  
\begin{itemize}
\item[(i)] $\|I'(u)\|_*\ge\bar\delta$ for all $u\in \mathcal C_*$ with $|I(u)-c|\le 2\bar\varepsilon$;
\item[(ii)] $\|u-K(u)\|\ge\bar\delta$ for all $u\in \mathcal C_*$ with $|I(u)-c|\le 2\bar\varepsilon$.
\end{itemize}
\end{lemma}
\begin{proof} (i) The proof follows by Lemma \ref{PalaisSmale}. Indeed, suppose by contradiction that (i) does not hold, then we can find a sequence $(u_n)\subset \mathcal C_*$ such that $\|I'(u_n)\|_*<\frac1n$ and $c-\frac1n \le I(u_n)\le c+\frac1n$ for all $n$. Hence, $(u_n)$ is a Palais-Smale sequence, and  since $I$ satisfies the Palais-Smale condition at level $c$, up to a subsequence, $u_n\to u$ in $W^{1,p}(B)$. Since $(u_n)\subset \mathcal C_*$ and $\mathcal C_*$ is closed, $u\in \mathcal C_*$. The fact that $I$ is of class $C^1$ then gives $I(u_n)\to c=I(u)$ and $I'(u_n)\to 0=I'(u)$, which contradicts the hypothesis.

(ii) Let $$I^{c+2\bar\varepsilon}_{c-2\bar\varepsilon}:= \{u\in \mathcal C_*\,:\,|I(u)-c|\le 2\bar\varepsilon\}.$$ By the part (i),  $I^{c+2\bar\varepsilon}_{c-2\bar\varepsilon}\subset W$, where $W$ is defined in Lemma \ref{KTI}. Hence, for all $u\in I^{c+2\bar\varepsilon}_{c-2\bar\varepsilon}$, $\|u-K(u)\|\ge\frac12\|u-\tilde T(u)\|$ by Lemma \ref{KTI}-(ii).
By the second inequality of \eqref{J2} and by (i), we have for all $u\in I^{c+2\bar\varepsilon}_{c-2\bar\varepsilon}$
$$\|u-\tilde T(u)\|\ge\frac{\bar\delta}{b(\|u\|+\|\tilde T(u)\|)^{p-2}}.$$

This implies by \eqref{salvezza}, that
$$\|u-\tilde T(u)\|\ge \frac{\bar\delta}{bC^{p-2}(1+\|u-\tilde T(u)\|^\beta)^{p-2}},$$
which in turn gives $\|u-\tilde T(u)\|\ge M$ for some positive $M$ and for all $u\in I^{c+2\bar\varepsilon}_{c-2\bar\varepsilon}$. Indeed, if by contradiction we had $\inf\|u-\tilde T(u)\|=0$ over all $u\in I^{c+2\bar\varepsilon}_{c-2\bar\varepsilon}$, we could find a sequence $(u_n)\subset I^{c+2\bar\varepsilon}_{c-2\bar\varepsilon}$ such that $\|u_n-\tilde T(u_n)\|\to 0$, and so by  passing to the limit as $n\to \infty$ in 
$$\|u_n-\tilde T(u_n)\|\ge \frac{\bar\delta}{bC^{p-2}(1+\|u_n-\tilde T(u_n)\|^\beta)^{p-2}},$$
we would have the contradiction $0\ge \bar\delta/(bC^{p-2})>0$, being $\beta>0$ thanks to the choice of $\ell$ in \eqref{defell}. Therefore, for all  $u\in I^{c+2\bar\varepsilon}_{c-2\bar\varepsilon}$,
$\|u-K(u)\|\ge\frac{M}2\ge\min\{\bar\delta, \frac{M}{2}\}$, still denoted by $\bar\delta$, and the proof is concluded. 
\end{proof}

\begin{lemma}[\textbf{Descending flow argument}]\label{deformation} Let $c\in\mathbb R$ be such that $I'(u)\neq 0$ for all $u\in \mathcal C_*$, with $I(u)=c$. Then, there exists a function $\eta:\mathcal C_*\to\mathcal C_*$ satisfying the following properties: 
\begin{itemize}
\item[(i)] $\eta$ is continuous with respect to the topology of $W^{1,p}(B)$;
\item[(ii)] $I(\eta(u))\le I(u)$ for all $u\in\mathcal C_*$;
\item[(iii)] $I(\eta(u))\le c-\bar\varepsilon$ for all $u\in\mathcal C_*$ such that $|I(u)-c|<\bar\varepsilon$;
\item[(iv)] $\eta(u)=u$ for all $u\in\mathcal C_*$ such that $|I(u)-c|>2\bar\varepsilon$,
\end{itemize}
where $\bar\varepsilon$ is the positive constant corresponding to $c$ given in Lemma \ref{conseqPS}.
\end{lemma}
\begin{proof} Let $\chi_1:\mathbb R\to [0,1]$ and $\chi_2: W^{1,p}(B)\to [0,1]$ be two smooth cut-off functions such that 
$$\chi_1(t)=
\begin{cases}1\quad&\mbox{if }|t-c|<\bar\varepsilon,\\
0&\mbox{if }|t-c|>2\bar\varepsilon,
\end{cases}
\qquad
\chi_2(u)=
\begin{cases}
1\quad&\mbox{if }\|u-K(u)\|\ge\bar\delta,\\
0&\mbox{if }\|u-K(u)\|\le\frac{\bar\delta}2,
\end{cases}
$$
where $\bar\delta$ and $\bar\varepsilon$ are given in Lemma \ref{conseqPS}.
Recalling the definition of $K$ in Lemma \ref{KTI}, let $\Phi: W^{1,p}(B)\to W^{1,p}(B)$ be the map defined by
$$\Phi(u):=\begin{cases}\chi_1(I(u))\chi_2(u)\frac{u-K(u)}{\|u-K(u)\|}\quad&\mbox{if }|I(u)-c|\le 2\bar\varepsilon,\\
0&\mbox{otherwise.}\end{cases}$$
Note that the definition of $\Phi$ is well posed by Lemma \ref{conseqPS}. For all $u\in\mathcal C_*$, we consider the Cauchy problem 
\begin{equation}\label{CauchyProblem}
\begin{cases}\frac{d}{dt}\eta(t,u(x))=-\Phi(\eta(t,u(x)))\quad&(t,x)\in(0,\infty)\times B,\\
\partial_\nu\eta(t,u(x))=0&(t,x)\in(0,\infty)\times\partial B,\\
\eta(0,u(x))=u(x) & x\in B.
 \end{cases}
\end{equation}
Being $K$ locally Lipschitz continuous by Lemma \ref{KTI}, for all $u\in\mathcal C_*$ there exists a unique solution $\eta(\cdot, u)\in C^1([0,\infty);W^{1,p}(B))$. 

We shall prove that for all $t>0$, $\eta(t,\mathcal C_*)\subset \mathcal C_*$. 
Fix $T>0$. For every $u\in\mathcal C_*$ and $n\in\N$ with $n\ge T/\bar\delta$, let
$$\begin{cases}\bar\eta_n(0,u):=u,\\
\bar\eta_n\left(t_{i+1},u\right):=\bar\eta_n\left(t_i,u\right)-\frac{T}n\Phi\left(\bar\eta_n\left(t_i,u\right)\right)\quad\mbox{for all }i=0,\dots,n-1,
\end{cases}
$$
with $$t_i:=i\cdot\frac{T}n\quad\mbox{for all }i=0,\dots,n.$$
Let us prove that for all $i=0,\dots,n-1$, $\bar\eta_n\left(t_{i+1},u\right)\in\mathcal C_*$. If $|I(u)-c|>2\bar\varepsilon$, then $\bar\eta_n\left(t_{i+1},u\right)=u\in\mathcal C_*$ for every $i=0,\dots,n-1$. Otherwise, let 
$$\lambda:=\frac{T}n\cdot\frac{\chi_1\left(I\left(\bar\eta_n\left(t_i,u\right)\right)\right)\chi_2\left(\bar\eta_n\left(t_i,u\right)\right)}{\|\bar\eta_n\left(t_i,u\right)-K\left(\bar\eta_n\left(t_i,u\right)\right)\|}.$$
Clearly, $\lambda\le1$ by Lemma \ref{conseqPS}-(ii), being $n\ge T/\bar\delta$. Therefore, it results for every $i=0,\dots,n-1$
$$\bar\eta_n\left(t_{i+1},u\right)=(1-\lambda)\bar\eta_n\left(t_i,u\right)+\lambda K\left(\bar\eta_n\left(t_i,u\right)\right)\in \mathcal C_*$$
by induction on $i$, by Lemma \ref{KTI}-(i), and by the convexity of $\mathcal C_*$. 
For every $i=0,\dots,n-1$, we can now define the line segment
$$\eta_n^{(i)}(t,u):= \left(1-\frac{t}Tn+i\right)\bar\eta_n\left(t_i,u\right)+\left(\frac{t}Tn-i\right)\bar\eta_n\left(t_{i+1},u\right)$$
for all $t\in\left[t_i,t_{i+1}\right]$.
We denote by $\eta_n:=\bigcup_{i=0}^{n-1}\eta^{(i)}_n$ the whole Euler polygonal defined in $[0,T]$. Being $\mathcal C_*$ convex, we get immediately that for all $t\in[0,T]$, $\eta_n(t,u)\in\mathcal C_*$. 

We claim that $\eta_n(\cdot,u)$ converges to the solution $\eta(\cdot,u)$ of the Cauchy problem \eqref{CauchyProblem} in $W^{1,p}(B)$.
Indeed, for all $i=0,\dots,n-1$, we integrate by parts the equation of \eqref{CauchyProblem} in the interval $[t_i,t_{i+1}]$ and we obtain
$$\eta(t_i+1,u)=\eta(t_i,u)-\frac{T}n\Phi(\eta(t_i,u))+\int_{t_i}^{t_{i+1}}(s-t_{i+1})\frac{d}{ds}\Phi(\eta(s,u))ds.$$
On the other hand, we define the error
$$\varepsilon_i:=\|\eta(t_i,u)-\eta_n(t_{i},u)\|\quad\mbox{for every }i=0,\dots,n.$$
Hence, for every $i=0,\dots,n-1$, we get 
\begin{equation}\label{error-estimate}
\varepsilon_{i+1}\le\varepsilon_i+\frac{T}n\|\Phi(\eta(t_i,u))-\Phi(\eta_n(t_i,u))\|+\left\|\int_{t_i}^{t_{i+1}}(t_{i+1}-s)\frac{d}{ds}\Phi(\eta(s,u))ds\right\|.
\end{equation}
Now, since $\Phi$ is locally Lipschitz and $\eta([0,T])\subset W^{1,p}(B)$ is compact, 
\begin{equation}\label{ineqonPhi}
\|\Phi(\eta(t_i,u))-\Phi(\eta_n(t_i,u))\|\le\varepsilon_i L_\Phi
\end{equation}
for some $L_\Phi=L_\Phi(\eta([0,T]))>0$. 
Furthermore, 
$$\begin{aligned}\left\|\int_{t_i}^{t_{i+1}}(t_{i+1}-s)\frac{d}{ds}\Phi(\eta(s,u))ds\right\|&\le\int_{t_i}^{t_{i+1}}(t_{i+1}-s)\left\|\frac{d}{ds}\Phi(\eta(s,u))\right\|ds\\
&\le\frac{T}{n}\int_0^T\|\Phi'(\eta(s,u))\|_*\|\Phi(\eta(s,u))\|ds\\
&\le\frac{T^2}{n}\sup_{s\in[0,T]}\|\Phi'(\eta(s,u))\|_*= \frac{T^2}{n}L_\Phi.
\end{aligned}$$
Thus, combining the last inequality with \eqref{ineqonPhi} and \eqref{error-estimate}, we have
$$\varepsilon_{i+1}\le\varepsilon_i+\frac{T}n\varepsilon_i L_\Phi+\frac{T^2}{n}L_\Phi\quad\mbox{for all }i=0,\dots,n-1.$$
This implies that 
$$\varepsilon_{i+1}\le \frac{T^2}n L_\Phi\sum_{j=0}^i\left(1+\frac{T}n L_\Phi\right)^j=T\left[\left(1+\frac{T}n L_\Phi\right)^{i+1}-1\right]\to 0\quad\mbox{as }n\to\infty,$$
where we have used the fact that $\varepsilon_0=0$. By the triangle inequality and the continuity of $\eta(\cdot,u)$ and $\eta_n(\cdot,u)$, this yields the claim. 

Hence, for all $t\in[0,T]$, $\eta(t,u)\in\mathcal C_*$ by the closedness of $\mathcal C_*$.   

For all $u\in\mathcal C_*$ and $t>0$ we can write
\begin{equation}\label{eq:flusso_decrescente}
\begin{aligned}I(\eta(t,u))-I(u)&=\int_0^t\frac{d}{ds}I(\eta(s,u))ds\\
&\hspace{-2.5cm}=-\int_0^t\frac{\chi_1(I(\eta(s,u)))\chi_2(\eta(s,u))}{\|\eta(s,u)-K(\eta(s,u))\|}I'(\eta(s,u))[\eta(s,u)-K(\eta(s,u))]ds\\
&\hspace{-2.5cm}
\le-\frac{a}2\displaystyle{\int_0^t\dfrac{\|\eta(s,u)-\tilde T(\eta(s,u))\|^{p}}{\|\eta(s,u)-K(\eta(s,u))\|}\chi_1(I(\eta(s,u)))\chi_2(\eta(s,u))ds}\le0,
\end{aligned}
\end{equation}
where we have used the inequality in Lemma \ref{KTI}-(iii).

Now, let $u\in\mathcal C_*$ be such that $|I(u)-c|<\bar\varepsilon$ and let $t\ge 2^{p+2}\bar\varepsilon/(a\bar\delta^{p-1})$. Then, two cases arise: either there exists $s\in[0,t]$ for which $I(\eta(s,u))\le c-\bar\varepsilon$ and so, by the previous calculation we get immediately that $I(\eta(t,u))\le c-\bar\varepsilon$, or for all $s\in[0,t]$, $I(\eta(s,u))> c-\bar\varepsilon$. In this second case, 
$$c-\bar\varepsilon< I(\eta(s,u))\le I(u)< c+\bar\varepsilon.$$
In particular, by Lemma \ref{conseqPS}-(i), $\eta(s,u)\in W$, by the definitions of $\chi_1$ and $\chi_2$, and by Lemma~\ref{conseqPS}-(ii), it results that for all $s\in[0,t]$ 
$$\chi_1(I(\eta(s,u)))=1,\qquad\|\eta(s,u)-K(\eta(s,u))\|\ge\bar\delta,\quad\mbox{and}\quad\chi_2(\eta(s,u))=1.$$
Hence, by \eqref{eq:flusso_decrescente} and Lemma \ref{KTI}-(ii) and (iii), we obtain
$$I(\eta(t,u))\le
I(u)-\displaystyle\int_0^t\frac{a}{2^{p+1}}\bar\delta^{p-1}ds
\le c+\bar\varepsilon-\frac{a}{2^{p+1}}\bar\delta^{p-1}t\le c-\bar\varepsilon.
$$

Finally, if we define with abuse of notation $$\eta(u):=\eta\left(\frac{2^{p+2}\bar\varepsilon}{a\bar\delta^{p-1}},u\right),$$ it is immediate to verify that $\eta$ satisfies (i)-(iv).
\end{proof}

\begin{lemma}[\textbf{Mountain pass geometry}]
\label{sec:case-multiple-fixed-2}
Let $\tau>0$ be such that $\tau <\min \{u_0-u_-,u_+-u_0\}$. 
Then there exists $\alpha>0$ such that 
\begin{itemize}
\item[(i)] $I(u)\ge I(u_-)+\alpha$ for every $u \in
\mathcal C_*$ with $\|u-u_-\|_{L^\infty(B)}=\tau$;
\item[(ii)] if $u_+< \infty$, then $I(u)\ge I(u_+)+\alpha$ for every $u \in
\mathcal C_*$ with $\|u-u_+\|_{L^\infty(B)}= \tau$.
\end{itemize}
\end{lemma}
\begin{proof} 
Suppose by contradiction that there exists a sequence $(w_n)_n 
\subset \mathcal C_*$ such that
\begin{equation}\label{wnbounded}
\|w_n\|_{L^\infty(B)}=w_n(1)=\tau>0\quad\mbox{for all }n
\end{equation}
and $\limsup \limits_{n \to \infty}
\bigl[I(u_-+w_n)-I(u_-)\bigr] \le 0$. Since
$$\begin{aligned}&\frac1p\int_B((u_-+w_n)^p-u_-^p)dx=\int_B\int_0^1(u_-+t w_n)^{p-1}w_n\,dtdx,\\
&\tilde F(u_-+w_n)-\tilde F(u_-)=\int_0^1\tilde f(u_-+tw_n)w_ndt,\end{aligned}$$
we get
\begin{align*}
I(u_-&+w_n)-I(u_-)\\ &= \frac{1}p\int_B\left(|\nabla w_n|^p+ m(u_-+w_n)^p -mu_-^p\right)\,dx 
 - \int_B \Bigl(\tilde F(u_- + w_n)-\tilde F(u_-))\,dx\\
&= \frac{1}p\int_B|\nabla w_n|^p\,dx + 
\int_B \int_0^1 \Bigl(m(u_-+t w_n)^{p-1} -  \tilde f(u_-+t w_n)\Bigr)w_n\,dt dx. 
\end{align*} 
Therefore, since by $(f_3)$ and the definition of $u_-$ 
\begin{equation}\label{fmp}
ms^{p-1}-\tilde f(s)>0 \qquad \text{for $s \in (u_-,u_0)$,}  
\end{equation}
we conclude that $\|\nabla w_n\|_{L^p}(B) \to 0$ and that $|\nabla w_n|\to0$ a.e. in $B$ up to a subsequence. Together with \eqref{wnbounded}, this ensures that $(w_n)$ is bounded in $W^{1,p}(B)$ and so, up to a subsequence, it is weakly convergent to some $w\in W^{1,p}(B)$. In particular, 
$$\lim_{n\to\infty}\int_B|\nabla (w_n-w)|^{p-2}\nabla (w_n-w)\cdot\nabla wdx=0.$$
By the Dominated Convergence Theorem, we now get that $\nabla w=0$ a.e. in $B$ and so the sequence
$(w_n)$ converges to the constant solution $w \equiv \tau$ in the
$W^{1,p}$-norm. Again by the Dominated Convergence Theorem we can conclude that 
\begin{align*}
0 &= \lim_{n \to
  \infty} \int_B \int_0^1 \Bigl(m(u_-+t w_n)^{p-1} -  \tilde f(u_-+t
w_n)\Bigr)w_n\,dtdx\\
&=  \int_B \int_0^1 \Bigl(m(u_-+t \tau)^{p-1} -  \tilde f(u_-+t
\tau)\Bigr)\tau\,dt dx,
\end{align*}
which contradicts \eqref{fmp}. Hence there exists $\alpha_1>0$ such that (i) holds. 

In a similar way, now using the fact that $ms^{p-1}-\tilde f(s)<0$ for $s \in (u_0,u_+)$, we find $\alpha_{2}>0$
such that (ii) holds if $u_+< \infty$. The claim then follows with $\alpha:= \min \{\alpha_1,\alpha_2\}$. 
\end{proof}
\smallskip

Let
\begin{equation}\label{eq:2}
\begin{aligned} U_- &:= \left\{u \in \mathcal C_* \::\: I(u)<I(u_-)+\frac{\alpha}{2},\:
\|u-u_- \|_{L^\infty(B)} < \tau\right\},\\
\\
U_+&:=\begin{cases}
\displaystyle{\left\{u \in \mathcal C_* \::\: I(u)<I(u_+)+\frac{\alpha}{2},\:
\|u-u_+ \|_{L^\infty(B)} < \tau\right\}},&\mbox{ if }u_+<\infty,\\
&\\
\left\{u \in \mathcal C_* \, :\, I(u)< I(u_-),\, \|u-u_-\|_{L^\infty(B)}>\tau\right\},&\mbox{ if }u_+=\infty
\end{cases}
\end{aligned}
\end{equation}
where $\tau$ and $\alpha$ are given by Lemma \ref{sec:case-multiple-fixed-2}, 
$$
\Gamma:=\left\{ \gamma\in C([0,1];\mathcal C_*)\ :\  \gamma(0) \in U_-,\:
  \gamma(1) \in U_+\right\},
$$
and
\begin{equation}\label{minmax}
c:=\inf_{\gamma\in\Gamma}\max_{t\in[0,1]} I(\gamma(t)).
\end{equation}

\begin{proposition}[\textbf{Mountain Pass Theorem}]\label{mountainpass} The value $c$ defined in \eqref{minmax} is finite and there exists a
critical point $u\in\mathcal C_*\setminus\{u_-,u_+\}$ of $I$ with $I(u)=c$. In particular, $u$ is a weak solution of \eqref{Pg}.
\end{proposition}
\begin{proof} We first observe that, by Lemma \ref{truncated}, any critical point of $I$ solves weakly \eqref{P} which is equivalent to \eqref{Pg}.
 
{\it Case $u_+<\infty$.} Pick any $\gamma\in\Gamma$. We note that by the definition of $U_-$ and $U_+$, and by the fact that $\tau<\min\{u_0-u_-,u_+-u_0\}$, we get $\|\gamma(1)-u_-\|_{L^\infty(B)}>\tau$ and $\|\gamma(0)-u_+\|_{L^\infty(B)}>\tau$. Now, since $\gamma$ is continuous with respect to the $W^{1,p}$-norm, by Lemma \ref{bounded} it is continuous also with respect to the $L^\infty$-norm. So, there exist $t_-,\,t_+\in(0,1)$ such that $\|\gamma(t_-)-u_-\|_{L^{\infty}(B)}=\tau$ and $\|\gamma(t_+)-u_+\|_{L^{\infty}(B)}=\tau$. Hence, by Lemma \ref{sec:case-multiple-fixed-2}, $I(\gamma(t_-))\ge I(u_-)+\alpha$ and $I(\gamma(t_+))\ge I(u_+)+\alpha$, which imply immediately that 
\begin{equation}\label{uu-u+}
c \ge \max \{I(u_-),I(u_+)\}+\alpha>\max \{I(u_-),I(u_+)\}.
\end{equation} 
On the other hand, $\Gamma$ is not empty, since it contains at least the path $t\in[0,1]\mapsto (1-t)u_-+tu_+$, hence $c<+\infty$. Therefore, $c$ is a finite number. 

Now, assume by contradiction that there does not exist a critical point $u\in\mathcal C_*$ for which $I(u)=c$. Then, there exists a deformation $\eta:\mathcal C_*\to\mathcal C_*$ satisfying (i)-(iv) of Lemma \ref{deformation}, with $\bar\varepsilon=\bar\varepsilon(c)>0$ given by Lemma \ref{conseqPS}. Without loss of generality, we assume that $4\bar\varepsilon<\alpha$. By the definition \eqref{minmax} of $c$, there exists a curve $\gamma\in\Gamma$ such that 
\begin{equation}\label{gammaok}\max_{t\in[0,1]}I(\gamma(t))<c+\bar\varepsilon
\end{equation} 
and we define the curve $\bar\gamma:[0,1]\to\mathcal C_*$, by $\bar\gamma(t):=\eta(\gamma(t))$. We check that also $\bar\gamma\in\Gamma$. Indeed, clearly $\bar\gamma\in C([0,1];\mathcal C_*)$. Moreover, $\bar\gamma(0)\in U_-$ because, by the definition of $U_-$ and by \eqref{uu-u+},
\begin{equation}\label{bargamma0}
I(\gamma(0))<I(u_-)+\frac\alpha2\le c-\alpha+\frac\alpha2<c-2\bar\varepsilon
\end{equation}
and so, by Lemma \ref{deformation}-(iv), $\bar\gamma(0)=\gamma(0)\in U_-$. Analogously, $\bar\gamma(1)=\gamma(1)\in U_+$. Furthermore, by \eqref{gammaok} and by Lemma \ref{deformation}-(iii), we get $I(\eta(\gamma(t)))\le c-\bar\varepsilon$ for all $t\in[0,1]$. Hence
$$\max_{t\in[0,1]}I(\bar\gamma(t))\le c-\bar\varepsilon,$$
which yields a contradiction with \eqref{minmax}. Finally, \eqref{uu-u+} ensures that the critical point $u\in\mathcal C_*$ at level $c$ cannot be  $u_-$ or $u_+$.
\smallskip

{\it Case $u_+=\infty$.} As in the previous case, for any $\gamma\in\Gamma$ there exists $t_-\in(0,1)$ for which $I(\gamma (t_-))\ge I(u_-)+\alpha$, and so we have 
\begin{equation}\label{cu-}
c \ge I(u_-)+\alpha>I(u_-).
\end{equation} 
Furthermore, for any $t>M$ we get by the fact that $\tilde f\in\mathfrak F$ 
\begin{equation}\label{I-infty}\begin{aligned}I(t\cdot 1)&=|B|\left(\frac{t^p}p-\int_0^t\tilde f(s)ds\right)\\
&\le |B|\left(\frac{t^p}p-\int_0^M\tilde f(s)ds-(m+\delta)\int_M^ts^{p-1}ds\right)\\
&\le\frac{|B|}p\left(t^p-pM\min_{s\in[0,M]}\tilde f(s)-(m+\delta)(t^p-M^p)\right)\\
&=C-\frac{|B|(m+\delta-1)}pt^p\to-\infty\quad\mbox{as }t\to\infty,
\end{aligned}
\end{equation}
being $m+\delta-1>0$. Hence, we can find a sufficiently large constant $k>0$ such that the curve
$$\gamma:\,t\in[0,1]\mapsto\, u_-+kt\in\mathcal C_*$$
is such that $\gamma(0)=u_-\in U_-$ and $\gamma(1)=u_-+k\in U_+$. Therefore, $\gamma\in\Gamma$ and consequently $c<+\infty$.
Now, suppose by contradiction that there does not exist any critical point $u\in\mathcal C_*$ of $I$, such that $I(u)=c$. Then, by Lemma \ref{conseqPS}-(i), $\|I'(u)\|_*\ge\bar\delta$ for any $u\in\mathcal C_*$ such that $|I(u)-c|\le 2\bar\varepsilon$. Without loss of generality, we can take $4\bar\varepsilon<\alpha$. In correspondence of $\bar\varepsilon$, consider the deformation $\eta$ built in Lemma \ref{deformation} and a curve $\gamma\in\Gamma$ such that $\max_{t\in[0,1]}I(\gamma(t))<c+\bar\varepsilon$. Now, let $\bar\gamma(t):=\eta(\gamma(t))$. We claim that $\bar\gamma\in\Gamma$. Indeed, $\bar\gamma(0)=\gamma(0)\in U_-$, since \eqref{bargamma0} holds by \eqref{cu-}.
Analogously, being $\gamma\in\Gamma$, $\gamma(1)\in U_+$ and so 
$$I(\gamma(1))< I(u_-)\le c-4\bar\varepsilon.$$
This yields, by Lemma \ref{deformation}-(iv), that $\eta(\gamma(1))=\gamma(1)$ and so $\bar\gamma\in\Gamma$. Now, since $I(\gamma(t))<c+\bar\varepsilon$, by Lemma \ref{deformation}-(iii), $I(\bar\gamma(t))\le c-\bar\varepsilon$ holds for all $t\in[0,1]$. This contradicts the definition of $c$. Hence, there exists a critical point $u\in\mathcal C_*$ of $I$ at level $c$, which is not equal to $u_-$ by \eqref{cu-}.
\end{proof}

\section{The mountain pass solution is non-constant}\label{sec4}
We are now ready to prove that the mountain pass solution $u\in\mathcal C_*\setminus\{u_-,\,u_+\}$ found in the previous section is nonconstant. To this aim, we observe that since in $\mathcal C_*$ the only constant solutions are $u_-$, $u_+$, and $u_0$, it remains to prove that $u\not\equiv u_0$.\smallskip
 
\begin{lemma}\label{4.9}
Let $v\in W^{1,p}(B)\setminus\{0\}$  be such that
\begin{equation}\label{v}
\int_B v dx=0,
\end{equation}
and let
$$
\psi: \R^2 \to \R,\qquad \psi(s,t):=I'(t(u_0+sv))[u_0+sv]. 
$$
There exist $\eps_1,\eps_2>0$ and a
$C^1$-function $h:(-\eps_1,\eps_1) \to (1-\eps_2,1+\eps_2)$
such that for $(s,t) \in V:= (-\eps_1,\eps_1) \times
(1-\eps_2,1+\eps_2)$ we have 
\begin{equation}\label{psi=0}
\psi(s,t)=0\quad \mbox{if and only if}\quad t=h(s).
\end{equation}
Moreover, 
\begin{itemize}
\item[(i)] $h(0)=1$, $h'(0)=0$;
\item[(ii)] $I(h(s)(u_0+sv))<I(u_0)$ for $s \in (-\eps_1,\eps_1)$, $s\neq 0$;
\item[(iii)] $\frac{\partial}{\partial t}\psi(s,t)<0$ for $(s,t)\in V$.
\end{itemize}
\end{lemma}

\begin{proof}
Since $I$ is a $C^2$-functional, $\psi$ is of class $C^1$ with $\psi(0,1)=0$. By $(f_3)$ we get 
\begin{equation}\label{psi'<0}
\frac{\partial}{\partial t}\Big|_{(0,1)}\psi(s,t) = I''(u_0)[u_0,u_0]= [m(p-1)u_0^{p-2}-\tilde f'(u_0)]\int_{B} 
u_0^2 \,dx <0
\end{equation}
and by \eqref{v}
$$
\frac{\partial}{\partial s}\Big|_{(0,1)}\psi(s,t) =I'(u_0)[v]+  I''(u_0)[u_0,v]=
[m(p-1)u_0^{p-2}-\tilde f'(u_0)]u_0 \int_{B} 
v\,dx =0.
$$
Thus the existence of $\eps_1,\eps_2$ and $h$, as well as property (i),
follow from the Implicit Function Theorem. To prove (ii), we write 
$h(s)= 1+o(s)$, for $s\in(-\varepsilon_1,\varepsilon_1)$, $s\neq 0$, so that 
$$
h(s)(u_0+sv)-u_0= sv + o(s)
$$
and therefore, by Taylor expansion and $(f_3)$,
\begin{align*}
I(h(s)(u_0+sv)) -I(u_0) &= \frac{1}{2} I''(u_0)[sv +
o(s),sv+o(s)]+o(s^2)\\&=\frac{s^2}{2} I''(u_0)[v,v]+o(s^2)\\
&=\frac{s^2}{2}\int_B[m(p-1)u_0^{p-2}-\tilde
  f'(u_0)]v^2\,dx+o(s^2)<0.
\end{align*}
Then, property (ii) holds after making $\eps_1$, $\eps_2$ smaller if necessary, and property (iii) is a consequence of \eqref{psi'<0} and of the regularity of $\psi$. 
\end{proof}

\begin{remark} Let $\mathcal N_*$ be the following Nehari-type set (see also \eqref{eq:M_q_def} ahead)
\begin{equation}\label{eq:nehari_def}
\mathcal N_*:=\{u\in \mathcal{C}_*\setminus\{0\}\,:\, I'(u)[u]=0\}.
\end{equation}
The previous lemma shows that $u_0$ is not a local minimum of the functional $I$ restricted to $\mathcal{N}_*$, that is to say, for every $\varepsilon>0$ there esists $u_\varepsilon \in \{ u\in W^{1,p}(B): \ \|u-u_0\|<\varepsilon \}\cap \mathcal N_*$  such that $I(u_\varepsilon)<I(u_0)$.

Indeed, if $v\in W^{1,p}(B)\setminus\{0\}$ is radial, nondecreasing and satisfies \eqref{v}, then for every  $s \in (-\eps_1,\eps_1)$, $s\neq 0$, it holds
\[
h(s)(u_0+sv)\in \mathcal N_*\quad\mbox{and}\quad  I(h(s)(u_0+sv))<I(u_0).
\]
Furthermore, since $h(s) \in C^1((-\eps_1,\eps_1))$ and $h(0)=1$, 
\[
\lim_{s\to0} \|h(s)(u_0+sv)-u_0\|=0,
\]
so that for every $\varepsilon>0$ there exists $s_\varepsilon \in (-\eps_1,\eps_1)$ such that $\|h(s_\varepsilon)(u_0+s_\varepsilon v)-u_0\|<\varepsilon$. The statement then follows with $u_\varepsilon=h(s_\varepsilon)(u_0+s_\varepsilon v)$.
\end{remark}

\begin{lemma}\label{lemma:nonconstant_p>2}
Fix 
$0<t_-<1<t_+$ such that  
\begin{equation}
  \label{eq:3}
t_- u_0 \in U_-,\quad t_+
u_0 \in U_+ \quad \text{and}\quad    u_- < t_- u_0 < u_0 < t_+ u_0 < u_+, 
\end{equation}
where $U_\pm$ are defined in $(\ref{eq:2})$. 
Let $v\in W^{1,p}(B)\setminus\{0\}$ radial, nondecreasing, satisfy \eqref{v}.
For $s  \ge 0$ define
\begin{equation}
  \label{eq:4}
\gamma_s: [t_-,t_+]\to W^{1,p}(B) \qquad \gamma_s(t):= t(u_0+sv).
\end{equation}
Then there exists $\bar s>0$ such that $\gamma_{\bar s}(t_\pm ) \in U_\pm$, $\gamma_{\bar s}(t) \in \mathcal C_*$ for $t_- \le t \le t_+$ and 
\begin{equation}
  \label{eq:1}
\max_{t_- \le t \le t_+} I(\gamma_{\bar s}(t))<I(u_0).  
\end{equation}
\end{lemma}

\begin{proof} 

{\it Case $u_+<\infty$.} First, we notice that such $t_-$ and $t_+$ exist. Indeed, by Lemma~\ref{sec:case-multiple-fixed-2} we know that $I(tu_0)\ge I(u_-)+\alpha$ for $t=(u_-+\tau)/u_0$. Hence, the continuity of $I$ implies that 
$$\exists \; t_-\in\left(\frac{u_-}{u_0},\frac{u_-+\tau}{u_0}\right)\quad\mbox{such that}\quad I(t_-u_0)<I(u_-)+\frac\alpha2.$$
The existence of $t_+$ can be proved analogously. 

We claim that there exists a positive constant $s_0\le \varepsilon_1$ ($\varepsilon_1$ as in Lemma~\ref{4.9}), such that 
\begin{equation}\label{claimstv}
I(\gamma_s(t)) <I(u_0)\quad\mbox{for all }(s,t) \in [-s_0,s_0]\times [t_-,t_+]\setminus\{(0,1)\}.
\end{equation}
We first observe that the function $t\in[t_-,t_+] \mapsto I(\gamma_0(t))=I(tu_0)$ has a
unique strict maximum point at $1$. Indeed, 
$$
\frac{d}{dt} I(\gamma_0(t)) =I'(t u_0)[u_0]=|B|(m(tu_0)^{p-1}-\tilde f(tu_0))u_0 
$$
and 
\begin{equation}\label{t-+}m(tu_0)^{p-1}-\tilde f(tu_0)\begin{cases}>0 \quad&\mbox{ if }t\in[t_-,1),\\
<0&\mbox{ if }t\in(1,t_+],\end{cases}\end{equation} 
since, being $\tilde f'(u_0)>m(p-1)u_0^{p-2}=m(u^{p-1})'|_{u=u_0}$, the inequalities \eqref{t-+} hold locally near $t=1$ and then, by \eqref{u-u+} and the definition of $t_-$ and $t_+$, they hold in the whole intervals $[t_-,1)$ and $(1,t_+]$, respectively. 
As a consequence, 
$$I(\gamma_0(t))<I(u_0)\quad \mbox{for all }t\in [t_-,1)\cup(1,t_+].$$
Now, by the continuity in $s$ of the function $I(t(u_0+sv))$, there exists $s_0\in(0,\varepsilon_1]$ such that 
\begin{equation}\label{fuoriV}
I(\gamma_s(t)) <I(u_0)\quad\mbox{for all }(s,t) \in [-s_0,s_0]\times [t_-,t_+] \setminus V,
\end{equation}
where we recall that $V=(-\varepsilon_1,\varepsilon_1)\times(1-\varepsilon_2,1+\varepsilon_2)$ as in the previous lemma.
On the other hand, if $(s,t)\in V$, by \eqref{psi=0} and Lemma~\ref{4.9}-(iii) for all $s\in(-\varepsilon_1,\varepsilon_1)$, we have
$$
\frac{d}{dt}I(\gamma_s(t))=\psi(s,t)
\begin{cases}>0\quad&\mbox{if }1-\varepsilon_2<t<h(s),\\
<0&\mbox{if }h(s)<t<1+\varepsilon_2.
\end{cases}
$$
Therefore, for all $s\in(-\varepsilon_1,\varepsilon_1)$, $h(s)$ is the unique maximum point of the map
$t\in(1-\varepsilon_2,1+\varepsilon_2)\mapsto I(\gamma_s(t))$, so that 
\begin{equation}\label{inV}
I(\gamma_s(t))\le I(\gamma_s(h(s)))<I(u_0)\quad\mbox{for all }(s,t)\in V\setminus\{(0,1)\}
\end{equation} 
by
Lemma~\ref{4.9}-(ii). By \eqref{fuoriV} and \eqref{inV}, the claim \eqref{claimstv} follows.

Furthermore, by (\ref{eq:3}) and since $v$ is radial and nondecreasing, we 
may choose $\bar s \in (0,s_0)$ so small that 
$$
\gamma_{\bar s}(t_-)=t_-(u_0+\bar sv) \in U_- \qquad \text{and} \qquad \gamma_{\bar s}(t_+)=t_+(u_0+\bar sv) \in U_+. 
$$
By the convexity of $\mathcal C_*$, for all $t\in[0,1]$
$$t\gamma_{\bar s}(t_-)+(1-t)\gamma_{\bar s}(t_+)=(u_0+\bar sv)[t_++t(t_--t_+)]=\gamma_{\bar s}(t_++t(t_--t_+))\in\mathcal C_*,$$
that is $\gamma_{\bar s}(t)\in\mathcal C_*$ for all $t\in[t_-,t_+]$, and we conclude the proof in this case.
\smallskip

{\it Case $u_+=\infty$.} The existence of $t_-$ follows as in the previous case, while the existence of $t_+$ is a consequence of the facts that $tu_0-u_->\tau$ for all $t>(u_-+\tau)/u_0$ and $I(tu_0)\to-\infty$ as $t\to+\infty$, see \eqref{I-infty}. The rest of the proof is analogous to case above, with the only change in the definition of $U_+$.
\end{proof}

\begin{proof}[$\bullet$ Proof of Theorem \ref{thm:main}]
By Proposition \ref{mountainpass}, there exists a mountain pass solution $u\in\mathcal C_*\setminus\{u_-,\,u_+\}$ of \eqref{Pg} such that $I(u)=c$. Furthermore, $u>0$ by \cite[Theorem~5]{Vazquez}. It only remains to prove that $u\not\equiv u_0$. To this aim, let $\gamma_{\bar s}$ be the curve given in Lemma \ref{lemma:nonconstant_p>2} and define $\bar\gamma(t):=\gamma_{\bar s}(t(t_+-t_-)+t_-)$ for all $t\in[0,1]$. Clearly, $\bar\gamma\in\Gamma$ and $c\le\max_{t\in [0,1]}I(\bar\gamma(t))<I(u_0)$ by the previous lemma. Hence, the mountain pass solution $u$ is different from the constant $u_0$. Since $u\in\mathcal C_*$, and the only constant solutions of \eqref{Pg} in $\mathcal C_*$ are $u_-,\, u_+,$ and $u_0$, this implies in particular that $u$ is nonconstant. 

The second part of the statement is proved by reasoning in the same way for each $u_{0,i}$, with $i=1,\dots,n$. We define $u_{\pm}^{(i)}$ and the cone of nonnegative, radial, nondecreasing functions $\mathcal C_*^{(i)}$, corresponding to each $u_{0,i}$. In this way, for every $i$, we get a nonconstant positive mountain pass solution $u^{(i)}\in\mathcal C_*^{(i)}$. Hence, $u_-^{(i)}\le u^{(i)}\le u_+^{(i)}$. Assume without loss of generality that $u_{0,1}<u_{0,2}<\dots<u_{0,n}$, then $u_-^{(1)}<u_+^{(1)}\le u_-^{(2)}<\dots\le u_+^{(n)}$ and so the $n$ solutions found are distinct.
\end{proof}

\begin{proof}[$\bullet$ Proof of Theorem \ref{thm:p=2}]
The proof of Theorem \ref{thm:main} works also for the case $p=2$, with the only exception of Lemma \ref{4.9}. In order to prove this lemma, we need the stronger assumption $(g_3')$ instead of $(g_3)$, and we can proceed as in \cite[Lemma~4.9]{BNW}.  
\end{proof}

\section{Asymptotic behavior in the pure power case}\label{sec5}
Let $q>p>2$. In this section we study the problem \eqref{Pg} for $g(u)=u^{q-1}$, namely
\begin{equation}\label{eq:pure_power}
\begin{cases}-\Delta_p u+u^{p-1}=u^{q-1}\quad&\mbox{in }B,\\
u>0 \quad&\mbox{in }B,\\
\partial_\nu u=0&\mbox{on }\partial B.
\end{cases}
\end{equation}

By Theorem \ref{thm:main} there exists a radial nondecreasing solution of \eqref{eq:pure_power} for every $q>p$. We remark that, concerning the notation in Sections \ref{sec2}-\ref{sec4}, in this specific case, $f=g$, $m=1$, $u_0=1$, $u_-=0$, $u_+=\infty$ and $\mathcal C_*=\mathcal C$. 
In this section we aim to find the asymptotic behavior of this solution of \eqref{eq:pure_power} as $q\to\infty$.

For all $q\ge p+1$, the functions $f_q(s):=s^{q-1}$ belong to the same set $\mathfrak F$ defined in \eqref{FmdeltaM}, with $m=1$ and $\delta=M-1$ for a fixed $M>1$, i.e. 
$$
f_q\in\mathfrak{F}=\left\{\varphi\in C([0,\infty))\,:\,\varphi\mbox{ nonnegative, }\varphi(s)\ge M s^{p-1}\,\mbox{for all }s\ge M\right\},\,q\ge p+1.
$$

For our analysis we need an additional property (namely \eqref{eq:f/s_increasing} below) on the truncated function $\tilde{f}$ introduced in Lemma \ref{truncated}; in order to ensure it, we provide here a more explicit construction of $\tilde{f}$.

\begin{lemma}\label{lemma:truncated2} 
For every $q\ge p+1$, there exists $\tilde{f}_q\in \mathfrak F\cap C^1([0,\infty))$ nondecreasing, satisfying $(g_1)$-$(g_3)$,
\begin{equation}\label{eq:f/s_increasing}
\mbox{fixed any } s>0, \quad\mbox{the map }t\in(0,\infty)\mapsto\frac{\tilde{f}_q(ts)}{t^{p-1}} \textrm{ is increasing},
\end{equation}
\begin{equation}\label{eq:subcritical}
\exists\,\ell\ \in\left(p,\min\left\{\frac{(p-1)^2+p-2}{p-2},p^*\right\}\right)\quad\mbox{such that }\quad\lim_{s\to\infty} \frac{\tilde{f}_q(s)}{s^{\ell-1}}=d,
\end{equation}
for some $d>0$, and with the property that if $u\in\mathcal C$ solves 
\begin{equation}\label{tildePpp}
\begin{cases}-\Delta_p u+u^{p-1}=\tilde{f}_q(u)\quad&\mbox{in }B,\\
u>0&\mbox{in }B,\\
\partial_\nu u=0&\mbox{on }\partial B,
\end{cases}
\end{equation}
then $u$ solves \eqref{eq:pure_power}. 
\end{lemma} 
\begin{proof} By Lemma \ref{apriori}, there exists $K_{\infty}$ such that $\|u\|_{L^\infty(B)}\leq K_{\infty}$ for every $u\in\mathcal{C}$ solution of \eqref{eq:pure_power}. Notice that $K_{\infty}\geq1$, because $1\in\mathcal{C}$ is a solution of \eqref{eq:pure_power} for every $q$.

Fix $s_0>\max\{K_{\infty},M\}$ and $\ell\in\left(p,\min\left\{\frac{(p-1)^2+p-2}{p-2},p^*\right\}\right)$.
We define
\begin{equation}\label{truncatedpurepower}
\tilde{f}_q(s):=\begin{cases}s^{q-1}\quad&\mbox{if }s\in[0,s_0],\\
s_0^{q-1}+\frac{q-1}{\ell-1}s_0^{q-\ell}(s^{\ell-1}-s_0^{\ell-1})&\mbox{otherwise}.\end{cases}
\end{equation}
It is straightforward to verify that $\tilde{f}_q$ is of class $C^1$, nonnegative and nondecreasing; it satisfies \eqref{eq:f/s_increasing} and \eqref{eq:subcritical}, with $d=(q-1)s_0^{q-\ell}/(\ell-1)$, which implies also $(g_2)$. Since $\tilde{f}_q(s)=s^{q-1}$ in $[0,1]$, $\tilde{f}_q$ also satisfies $(g_1)$ and $(g_3)$.

In order to prove that $\tilde f_q\in \mathfrak F$, it remains to show that 
\begin{equation}\label{uff}\tilde f_q(s)\ge M s^{p-1}\quad\mbox{for all }s\ge M.\end{equation}
Since $\tilde f_q=f_q$ in $[M,s_0]$, it is enough to verify that \eqref{uff} for all $s> s_0$.
This is equivalent to show that  
\begin{equation}\label{eq:xi}
\xi(s):=[(q-1)s_0^{q-\ell}s^{\ell-p}-M(\ell-1)]s^{p-1}\geq (q-\ell) s_0^{q-1} \quad \textrm{for all } s> s_0.
\end{equation}
Now, $\xi(s_0)>(q-\ell) s_0^{q-1}$, being $s_0>M$. Moreover, $\xi'(s)\ge 0$ for all $s\ge s_0$. Therefore, \eqref{eq:xi} holds and so, by Lemma \ref{apriori}, all solutions of \eqref{tildePpp} solve also \eqref{eq:pure_power}.
\end{proof}

We denote by $\tilde F_q$ the primitive of $\tilde f_q$ and by $I_q$ the associated energy functional.

We introduce the Nehari-type set
\begin{equation}\label{eq:M_q_def}
\mathcal N_q:=\left\{u\in\mathcal C\setminus\{0\}\,: \int_B(|\nabla u|^p+|u|^p)dx=\int_B\tilde f_q(u)udx\right\}.
\end{equation}

\begin{lemma}\label{palla_nella_Nehari}
There exists $\sigma>0$ such that $$\inf_{q\ge p+1}\inf_{u\in\mathcal N_q}\|u\|_{L^\infty(B)}\ge\sigma.$$
\end{lemma}
\begin{proof}
Suppose by contradiction that there exist $(q_n)$, with $q_n\ge p+1$ for any $n$, and $(u_n)\subset\mathcal N_{q_n}$ such that $\|u_n\|_{L^\infty(B)}\to 0$ as $n\to\infty$. 
Then, for $n$ sufficiently large $\tilde f_{q_n}(u_n)=u_n^{q_n-1}$ and, being ${q_n}\ge p+1$, there exists $\varepsilon>0$ such that $\tilde f_{q_n}(u_n)u_n=u_n^{q_n}\le(1-\varepsilon)u_n^p$ for every $n$. Therefore, since $u_n\in\mathcal N_{q_n}$, for $n$ large we get 
$$\begin{aligned}
0&=\int_B(|\nabla u_n|^p+u_n^p-u_n^{q_n})dx\ge\int_B(|\nabla u_n|^p+u_n^p-(1-\varepsilon)u_n^p)dx\\
&=\int_B(|\nabla u_n|^p+\varepsilon u_n^p)dx\ge0,
\end{aligned}$$
which is impossible, since $0\not\in\mathcal N_{q_n}$. 
\end{proof}

Let $c_q$ be the mountain pass level corresponding to $q$ as in \eqref{minmax}, that is to say
\begin{equation}
c_q=\inf_{\gamma\in\Gamma_q}\max_{t\in[0,1]} I_q(\gamma(t)),
\end{equation}
where
$$
\Gamma_q:=\left\{ \gamma\in C([0,1];\mathcal C)\ :\  \gamma(0) \in U_{q,-},\: \gamma(1) \in U_{q,+}\right\},
$$
and
\begin{equation}
\begin{aligned} 
U_{q,-} &= \left\{u \in \mathcal C \::\: I_q(u)<\frac{\alpha_q}{2},\:
\|u\|_{L^\infty(B)} < \tau\right\},\\
U_{q,+}&= \left\{u \in \mathcal C \, :\, I_q(u)< 0,\, \|u\|_{L^\infty(B)}>\tau\right\},
\end{aligned}
\end{equation}
with $\tau<\min\{\sigma, 1\}$, $\sigma$ given in Lemma \ref{palla_nella_Nehari}, and $\alpha_q$ is given as in Lemma \ref{sec:case-multiple-fixed-2} for $g(s)=s^{q-1}$.\smallskip

We notice that property \eqref{eq:f/s_increasing} is crucial for the proof of the following lemma.

\begin{lemma}\label{gH} For every $u\in\mathcal C\setminus\{0\}$ there exists a unique $h_q(u)>0$ such that $h_q(u)u\in\mathcal{N}_q$. It holds 
\begin{equation}\label{Iqhq}
I_q(tu)>0 \quad\mbox{for all }t\in(0,h_q(u)].
\end{equation}
Furthermore, if $(u_n)\subset \mathcal C\setminus\{0\}$ is such that $u_n\to u\in \mathcal C\setminus\{0\}$ with respect to the $W^{1,p}$-norm, then $h_q(u_n)\to h_q(u)$. Finally, the map 
$$H:u\in \mathcal C\cap\mathcal S^1\mapsto h_q(u)u\in\mathcal N_q,\quad\mbox{where }\mathcal S^1:= \{u\in W^{1,p}(B)\,:\,\|u\|=1\}$$ is a homeomorphism.
\end{lemma}
\begin{proof} Fix $u\in\mathcal C\setminus\{0\}$ and consider corresponding map $\phi:t\in[0,\infty)\mapsto I_q(tu)$. Clearly, for all $t>0$, it results that 
\begin{equation}\label{g'}\begin{aligned}\phi'(t)&=0\; \Longleftrightarrow\; I_q'(tu)[u]=0\;\Longleftrightarrow\;\|u\|^p=\frac1{t^{p-1}}\int_B\tilde{f}_q(tu)udx.
\end{aligned}\end{equation}
Furthermore, $\phi(0)=0$ and, since $1<p<\ell<q$,
\begin{equation}\label{g<>0}
\begin{aligned}
\phi(t)&=\frac{t^p}p\|u\|^p-\frac{t^q}q\|u\|^q_{L^q(B)}>0\quad\mbox{for $t>0$ small,}\\
\phi(t)&=\frac{t^p}p\|u\|^p+t\frac{q-\ell}{\ell-1}s_0^{q-1}\|u\|_{L^1(B)}-t^\ell\frac{q-1}{\ell(\ell-1)}s_0^{q-\ell}\|u\|_{L^\ell(B)}<0\quad\mbox{for $t$ large.} 
\end{aligned}
\end{equation}
Therefore, by the continuity of $\phi$, there exists $h_q(u)\in(0,\infty)$ such that $$\phi(h_q(u))=\max_{t\in[0,\infty)}\phi(t)$$ and consequently, $\phi'(h_q(u))=0$. We can prove that the maximum point $h_q(u)$ is the unique non-zero critical point of $\phi$. Indeed, suppose by contradiction that $\phi$ admits another critical point $0<\bar h\neq h_q(u)$, then by \eqref{g'}
$$\|u\|^p=\frac1{\bar h^{p-1}}\int_B\tilde f_q(\bar h u)udx=\frac1{h_q(u)^{p-1}}\int_B \tilde f_q(h_q(u) u)udx,$$
which contradicts \eqref{eq:f/s_increasing}, being $\tilde f_q(ts)/t^{p-1}$ strictly increasing in $t$, for all $s>0$. Thus, $h_q(u)$ is unique and \eqref{Iqhq} holds. Furthermore, $H$ is well defined. 

Now, let $(u_n)\subset\mathcal C\setminus\{0\}$, $u_n\to u\in\mathcal C\setminus\{0\}$. 
Suppose by contradiction that the corresponding sequence $(h_q(u_n))$ is unbounded. Then, 
$$\|u_n\|^p=\frac1{h_q(u_n)^{p-1}}\int_B\tilde f_q(h_q(u_n)u_n)u_ndx\nearrow\infty\quad\mbox{as }n\to\infty,$$ 
by \eqref{eq:f/s_increasing}. This contradicts the fact that $(u_n)$ is convergent. Hence, $(h_q(u_n))$ is bounded and we can find a subsequence, still indexed by $n$, for which $h_q(u_n)\to \bar h$. For the Dominated Convergence Theorem we obtain
$$\|u_n\|^p=\frac1{h_q(u_n)^{p-1}}\int_B\tilde f_q(h_q(u_n)u_n)u_ndx\to \frac1{\bar h^{p-1}}\int_B\tilde f_q(\bar h u)udx.$$ 
By the uniqueness of the limit, this yields
$$\frac1{\bar h^{p-1}}\int_B\tilde f_q(\bar hu)udx=\|u\|^p,$$
that is $\bar h=h_q(u)$ and in particular $H$ is continuous. 

Finally, the continuous map $v\in\mathcal N_q
\mapsto v/\|v\|\in\mathcal C\cap\mathcal S^1$ is the inverse of $H$, by the uniqueness of $h_q(u)$ and by the fact that $h_q(u)=1$ if and only if $u\in\mathcal N_q$.
\end{proof}

The preceding lemma allows to prove that the mountain pass level in the cone coincides with a Nehari-type level in the cone. 

\begin{lemma}\label{lemma:nehari}
The following equalities hold
\begin{equation}\label{c_uguali}
c_q=\inf_{u\in\mathcal C\setminus\{0\}}\sup_{t\ge0}I_q(tu)=\inf_{u\in\mathcal N_q}I_q(u).
\end{equation}
\end{lemma}
\begin{proof} 
We shall split the proof of \eqref{c_uguali} into three steps.

{\it Step 1.} We first prove that $\inf_{u\in\mathcal C\setminus\{0\}}\sup_{t\ge0}I_q(tu)=\inf_{u\in\mathcal N_q}I_q(u)$. From Lemma \ref{gH}, we know that 
$$\inf_{u\in\mathcal C\setminus\{0\}}\sup_{t\ge0}I_q(tu)=\inf_{u\in\mathcal C\setminus\{0\}}I_q(h_q(u)u)\ge \inf_{u\in\mathcal N_q}I_q(u)$$
being $h_q(u)u\in\mathcal N_q$. On the other hand, 
$$\inf_{u\in\mathcal C\setminus\{0\}}\sup_{t\ge0}I_q(tu)\le\inf_{u\in\mathcal C\cap\mathcal S^1}\sup_{t\ge0}I_q(tu)=\inf_{u\in\mathcal C\cap\mathcal S^1}I_q(h_q(u)u)=\inf_{u\in\mathcal N_q}I_q(u),$$
where we have used the fact that $H$ defines a homeomorphism between $\mathcal C\cap\mathcal S^1$ and $\mathcal N_q$.\smallskip

{\it Step 2.} Now we prove that $c_q\le \inf_{u\in\mathcal C\setminus\{0\}}\sup_{t\ge0}I_q(tu)$. Indeed, for all $u\in\mathcal C\setminus\{0\}$, by \eqref{g<>0}, there exists $\bar t_u$ so large that $I_q(\bar t_u u)<0$ and $\|\bar t_u u\|_{L^\infty(B)}>\tau$. Hence, we can consider the curve $\gamma: t\in[0,1]\mapsto t\bar t_u u\in\mathcal C$. Clearly $\gamma\in\Gamma_q$, so that we get 
$$c_q\le\inf_{u\in\mathcal C\setminus\{0\}}\max_{t\in[0,1]}I_q(t\bar t_u u)\le  \inf_{u\in\mathcal C\setminus\{0\}}\sup_{t\ge0}I_q(t\bar t_u u)=\inf_{u\in\mathcal C\setminus\{0\}}\sup_{t\ge0}I_q(tu).$$
\smallskip

{\it Step 3.} Finally we show that $c_q\ge \inf_{u\in\mathcal N_q}I_q(u)$. Let $\gamma$ be any curve in $\Gamma_q$, we claim that $\gamma([0,1])\cap\mathcal N_q\neq\emptyset$. If the claim holds true, we know that for any $\gamma\in \Gamma_q$ there exists $t_\gamma\in[0,1]$ such that $\gamma(t_\gamma)\in\mathcal N_q$, and so we can conclude that
$$c_q\ge\inf_{\gamma\in\Gamma_q} I_q(\gamma(t_\gamma))\ge\inf_{u\in\mathcal N_q}I_q(u).$$
It remains to prove the claim. Pick any $\gamma\in\Gamma_q$, then $\|\gamma(0)\|_{L^\infty(B)}<\tau<\sigma$, with $\sigma$ given in Lemma \ref{palla_nella_Nehari}. If $\gamma(0)\not\equiv 0$, by Lemma \ref{gH} we know that there exists a unique $h_q(\gamma(0))>0$ such that $h_q(\gamma(0))\gamma(0)\in\mathcal N_q$. Hence, together with Lemma \ref{palla_nella_Nehari} we obtain
$$\sigma h_q(\gamma(0))>\tau h_q(\gamma(0))>\|h_q(\gamma(0))\gamma(0)\|_{L^\infty(B)}\ge \sigma,$$
so that $h_q(\gamma(0))>1$. 
If $\gamma(0)\equiv 0$, let $\varepsilon\in(0,\tau)$. By the continuity of $\gamma$ in the $L^\infty$-norm (see Lemma \ref{bounded}) and the fact that $\|\gamma(1)\|_{L^\infty(B)}>\tau$, there exists $\bar t\in(0,1)$ such that $\|\gamma(\bar t)\|_{L^{\infty}(B)}=\varepsilon$. Then, by Lemma \ref{gH}, there is a unique $h_q(\gamma(\bar t))$ for which $h_q(\gamma(\bar t))\gamma(\bar t)\in\mathcal N_q$. Proceding as in the case $\gamma(0)\not\equiv 0$, by replacing 0 by $\bar t$, we get immediately that $h_q(\gamma(\bar t))>1$.

Furthermore, by \eqref{Iqhq}, $I_q(t\gamma(1))>0$ for all $t\in(0,h_q(\gamma(1))]$. Suppose by contradiction that $h_q(\gamma(1))\ge1$, then $I_q(\gamma(1))>0$, but this is absurd, being $\gamma\in\Gamma_q$. 

In conclusion, we have that there exists $\bar t\in[0,1)$ for which $h_q(\gamma(\bar t))>1$, and that $h_q(\gamma(1))<1$. By the continuity of $h_q$ proved in Lemma \ref{gH}, and the continuity of $\gamma$, there exists $t_\gamma\in(\bar t,1)$ for which $h_q(\gamma(t_\gamma))=1$, that is $\gamma(t_\gamma)\in\mathcal N_q$.   
\end{proof}

By Theorem \ref{thm:main} there exists a nonconstant, nondecreasing, radial solution $u_q$ of \eqref{eq:pure_power}, which by Proposition \ref{mountainpass} can be caracterized as a mountain pass solution, that is to say 
\begin{equation}\label{eq:u_q_def}
c_q=I_q(u_q) \quad\textrm{and}\quad I_q'(u_q)=0.
\end{equation}
We shall now provide some a priori bounds on $u_q$, uniform in $q$.
 
\begin{lemma}\label{commonbound} There exists $C>0$ independent of $q$ such that, for all $q\ge p+1$,
$$\|u_q\|_{C^1(\bar B)}\le C.$$ 
\end{lemma}
\begin{proof}
By integrating the equation satisfied by $u_q$, we get
\[
\int_B u_q^{p-1}(1-u_q^{q-p}) dx =0.
\]
Since $u_q\not\equiv1$ is positive and nondecreasing, we deduce that
\begin{equation}\label{eq:u(0)}
u_q(0)<1, \quad u_q(1)>1\quad\mbox{for all }q\ge p+1.
\end{equation}

Consider the equation satisfied by $u_q$ in radial form. We multiply it by $u_q'\ge 0$ to obtain
\[
\left( \frac{p-1}{p} (u_q')^p +\frac{u_q^q}{q} - \frac{u_q^p}{p} \right)'
=-\frac{N-1}{r} (u_q')^p.
\]
We deduce that the function
\begin{equation}
L_q(r):=\frac{p-1}{p} (u_q'(r))^p - \frac{u_q(r)^p}{p} +\frac{u_q(r)^q}{q},\quad r\in[0,1]
\end{equation}
is nonincreasing in $r$, and hence, using \eqref{eq:u(0)},
\[
L_q(r)\leq L_q(0) = -\frac{u_q(0)^p}{p}+\frac{u_q(0)^q}{q} \leq 0\quad\mbox{for all } r\in[0,1].
\]
We note that $L_q(r)\le 0$ is equivalent to 
$$(u_q(r), u'_q(r))\in\Sigma:=\left\{(x,y)\in\mathbb R^2\,:\,x\ge0,\, 0\le y\le \left[\frac{p}{p-1}\left(\frac{x^p}p-\frac{x^q}q\right)\right]^{1/p} \right\}.$$

\begin{figure}
\centering
\includegraphics[scale=0.8]{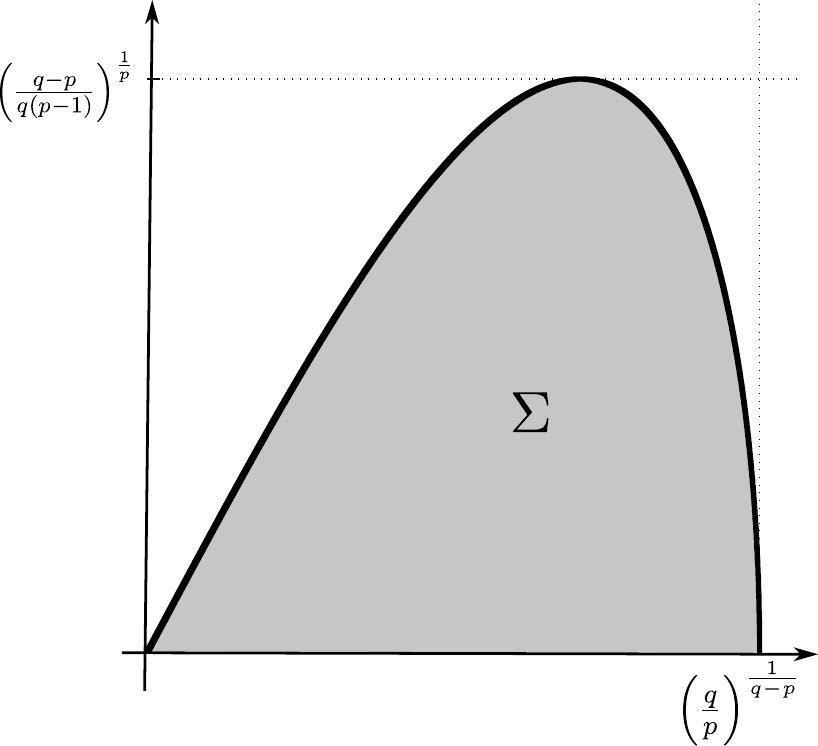}
\caption{The graphic of the function $y=\left[\frac{p}{p-1}\left(\frac{x^p}p-\frac{x^q}q\right)\right]^{1/p}$ for $x,\,y\ge0$.}\label{fig:1}
\end{figure}
This implies (see Figure \ref{fig:1})
\[
u_q \leq \left(\frac{q}{p}\right)^\frac{1}{q-p} \to 1 \quad \textrm{as } q\to\infty,
\]
\[
u_q' \leq \left(\frac{q-p}{q(p-1)}\right)^\frac{1}{p} \to p^{-\frac{1}{p}} \quad \textrm{as } q\to\infty. \qedhere
\]
\end{proof}

The previous a priori bounds ensure the existence of a limit profile.

\begin{lemma}\label{weakconv} 
There exists a function $u_\infty\in \mathcal C$ for which 
$$u_q\rightharpoonup u_\infty\;\mbox{ in }W^{1,p}(B), \quad u_q\to u_\infty \;\mbox{ in }C^{0,\nu}(\bar B) \quad\mbox{as }q\to\infty,$$
for any $\nu\in(0,1)$. Furthermore, $u_\infty(1)=1$. 
\end{lemma}
\begin{proof}
The existence of $u_\infty$ and the convergence are consequences of the previous lemma, together with the compactness of the embedding $C^1\hookrightarrow C^{0,\nu}$. Since, up to a subsequence, $u_q \to u_\infty$ pointwise, we deduce that $u_\infty \in \mathcal{C}$. From \eqref{eq:u(0)} we immediately get that $u_\infty(1)\geq 1$.

It only remains to show that $u_\infty(1)= 1$. To this aim, suppose by contradiction that $u_\infty(1)> 1$. Then there exist $s\in(0,1)$ and $\delta>0$ such that
\begin{equation}\label{eq:u_infty(1)_contrad}
u_q(r)\geq 1+\delta\quad \textrm{ for every } s\leq r\leq 1,
\end{equation}
and for every $q$ sufficiently large.
We integrate the equation satisfied by $u_q$ in the interval $(s,1)$
\[
s^{N-1} (u_q'(s))^{p-1} = \int_s^1 u_q^{p-1} (u_q^{q-p}-1) r^{N-1} dr
\]
and we replace \eqref{eq:u_infty(1)_contrad} to obtain
\[
s^{N-1} (u_q'(s))^{p-1} \geq \int_s^1 (1+\delta)^{p-1} ((1+\delta)^{q-p}-1) r^{N-1} dr \to +\infty 
\]
as $q\to\infty$, in contradiction with Lemma \ref{commonbound}.
\end{proof}

\begin{lemma}\label{cinfinito}
The quantity
$$
c_\infty:=\inf\left\{\frac{\|v\|^p}p\,:\,v\in\mathcal C,\,v=1\mbox{ on }\partial B\right\}
$$
is achieved by the unique radial function $G$ satisfying \eqref{eqforG}.
\end{lemma}
\begin{proof} 
By the Direct Method of the Calculus of Variations, 
$$c'_\infty:=\inf\left\{\frac{\|v\|^p}p\,:\,v\in W^{1,p}(B),\,v=1\mbox{ on }\partial B\right\}
$$ 
is uniquely achieved by $G$. Let us prove that $c_\infty=c_\infty'$. Clearly, $c'_\infty\le c_\infty$. 
On the other hand, by the comparison principle, $G>0$ in $B$, and by the radial symmetry $G'(0)=0$. If we integrate the equation in \eqref{eqforG} in its radial form, we get
$$r^{N-1}|G'(r)|^{p-2}G'(r)=\int_0^r t^{N-1}G(t)^{p-1}dt>0.$$
Hence, $G\in\mathcal C$ and so $c_\infty\le c'_\infty$. This concludes the proof.  
\end{proof}

\begin{lemma} \label{lemma:c_infty<c_q} It holds
$c_\infty\le\liminf_{q\to\infty}c_q$.
\end{lemma}
\begin{proof} We take the function $u_\infty$ introduced in Lemma \ref{weakconv} as test function in the definition of $c_\infty$ and  we get for some constant $C>0$ independent of $q$,
\begin{equation}\label{eq:c_infty<c_q}
\begin{split}
c_\infty & \le\frac{\|u_\infty\|^p}p \leq 
\liminf_{q\to\infty} \frac{\|u_q\|^p}{p}
=\liminf_{q\to\infty} \left( I_q(u_q) +\int_B \frac{u_q^{q}}{q}dx \right) \\
& = \liminf_{q\to\infty} \left( c_q +\frac{\|u_q\|^p}{q} \right)
\leq  \liminf_{q\to\infty} \left( c_q + \frac{C}{q}\right)= \liminf_{q\to\infty} c_q,
\end{split}
\end{equation}
where we used \eqref{eq:pure_power}, \eqref{eq:u_q_def}, and Lemma \ref{commonbound}.
\end{proof}

\begin{proof}[$\bullet$ Proof of Theorem \ref{thm:asymptotic_q}]
Let $G$ be the unique solution of \eqref{eqforG}. Since $G\in\mathcal{C}\setminus\{0\}$, by Lemma \ref{gH} there exists a unique $h_q(G)>0$ such that $h_q(G)G \in \mathcal{N}_q$. 
Since $\tilde{f}_q(s)=s^{q-1}$ for $s\leq 1=\|G\|_{L^\infty(B)}$, we have
\begin{equation}\label{eq:t_to1}
h_q(G)= \left( \frac{\|G\|^p}{\int_B G^qdx} \right)^{\frac{1}{q-p}}
\to\frac{1}{\|G\|_{L^\infty(B)}}=1 \quad\textrm{ as } q\to\infty.
\end{equation}
This implies
$$
c_\infty=\frac{\|G\|^p}{p} 
=\lim_{q\to\infty} \frac{\|h_q(G) G\|^p}{p}
=\lim_{q\to\infty} \left(I_q(h_q(G)G)+\frac{h_q(G)^q}{q}\int_B G^q dx \right).
$$
Now, since $h_q(G)G\in\mathcal{N}_q$, we can rewrite the last term as
\[
c_\infty=\lim_{q\to\infty}\left( I_q(h_q(G)G)+ \frac{\|h_q(G) G\|^p}{q} \right) = \lim_{q\to\infty} I_q(h_q(G)G),
\]
by \eqref{eq:t_to1}. Then, since $h_q(G)G\in\mathcal{N}_q$, Lemma \ref{lemma:nehari} implies that 
$$c_q=\inf_{u\in\mathcal N_q}I_q(u)\le I_q(h_q(G)G).$$
The previous two equations provide 
$c_\infty \geq \limsup_{q\to\infty} c_q$.
By combining this inequality with Lemma~\ref{lemma:c_infty<c_q}, we obtain that 
\begin{equation}
c_\infty=\lim_{q\to\infty}c_q.
\end{equation}
As a consequence, the inequalities in \eqref{eq:c_infty<c_q} are indeed equalities, so that 
\[
\lim_{q\to\infty} \|u_q\|=\|G\|\quad\mbox{and}\quad \|u_\infty\|=\|G\|.
\]
Hence, $u_\infty$ achieves $c_\infty$ and, by Lemma \ref{cinfinito}, $u_\infty=G$.
Together with the $W^{1,p}$-weak convergence and the uniform convexity of $W^{1,p}(B)$, this implies that $u_q\to G$ in $W^{1,p}(B)$. By Lemma \ref{weakconv} the convergence is also $C^{0,\nu}(\bar B)$ for any $\nu\in(0,1)$.
\end{proof}

\appendix

\section{Some remarks in the case $1<p<2$}\label{A}
In this appendix we consider the case $1<p<2$. We prove that Proposition \ref{mountainpass} holds under an additional assumption on $g$, that is to say, a mountain pass solution exists also in this case. Nevertheless, we do not know whether the mountain pass solution is nonconstant. In particular, we prove that Lemma \ref{4.9}-(ii) does not hold for $1<p<2$ and $g(u)=u^{q-1}$.

We require $g$ to be of class $C^1((0,\infty))\cap C([0,\infty))$, to satisfy $(g_1)$-$(g_3)$ and 

\begin{itemize}
\item[$(g_4)$] $\inf\{ t\in (u_0,\infty): \ g(t)=t^{p-1}\} < \infty$.
\end{itemize}
We remark that the assumption on the regularity of $g$ is slightly weaker than in case $p\ge2$. This allows us to cover the nonlinearities which behaves like $s^{q-1}$ ($q>p$) near the origin. 

The results in Section \ref{sec2} hold also in this setting with exactly the same proofs. The only difference is that the function $f$ in Lemma \ref{gtof} is of class $C^1((0,\infty))\cap C([0,\infty))$ as $g$. 

Furthermore, proceeding as in Lemma~\ref{truncated} we can build, also in this case, the subcritical nonlinearity $\tilde f$.

\begin{lemma}\label{truncated1} For every $\ell\in(p,p^*)$, there exists $\tilde{f}\in \mathfrak F$ nondecreasing, satisfying
$$\tilde f=f\quad\mbox{in }[0,s_0]\mbox{ for some }s_0>\max\{K_\infty,M\}$$
($K_\infty$ as in Lemma \ref{apriori} and $M$ as in \eqref{Mdelta}),  $(f_1)$-$(f_3)$, 
 \begin{equation}\label{subcritical1}
\lim_{s\to\infty}\frac{\tilde{f}(s)}{s^{\ell-1}}=1 ,
\end{equation}
and with the property that, if $u\in\mathcal C$ solves \eqref{tildeP}, then $u$ solves \eqref{P}.  
\end{lemma}

The associated energy functional $I$ is defined as in \eqref{I} and is of class $C^{1}$. 

In this setting there exists a mountain pass solution of the problem, 
as stated in the following proposition. 
 
\begin{proposition}[\textbf{Mountain Pass Theorem}]\label{mountainpass1} 
Let $1<p<2$. Let $g \in C^1((0,\infty)) \cap C([0,\infty))$ satisfy $(g_1)$-$(g_4)$.
Then the value $c$ defined in \eqref{minmax} is finite and there exists a critical point $u\in\mathcal C_*$ of $I$ with $I(u)=c$.
\end{proposition}

The proof of Proposition \ref{mountainpass1} relies on several preliminary results, which hold under the same assumptions. 

\begin{proposition}\label{Ttildecompact1}
The operator $\tilde T$ defined in \eqref{eq:tildeT_def} is compact. Furthermore, there exist two positive constants $a,\,b$ such that for all $u\in W^{1,p}(B)$ the following properties hold
\begin{equation}\label{J1}
\begin{aligned}
&I'(u)[u-\tilde T(u)]\ge a\|u-\tilde T(u)\|^2(\|u\|+\|\tilde T(u)\|)^{p-2},\\
&\|I'(u)\|_{*}\le b\|u-\tilde T(u)\|^{p-1},
\end{aligned}
\end{equation}
\end{proposition}
\begin{proof}
The proof is analogous to the one of Proposition \ref{Ttildecompact}.
\end{proof}

\begin{lemma}[\textbf{Palais-Smale condition}]
$I$ satisfies the Palais-Smale condition, i.e. every sequence $(u_n)\subset W^{1,p}(B)$ such that $(I(u_n))$ is bounded and $I'(u_n)\to0$ in $(W^{1,p}(B))'$ admits a convergent subsequence. 
\end{lemma}
\begin{proof}
Reasoning as in Lemma \ref{PalaisSmale}, we obtain that any (PS)-sequence $(u_n)$ is weakly converging to some $u$ in $W^{1,p}(B)$ and that  \eqref{trineq} holds.
Now, by the first inequality of \eqref{J1} we get 
$$\|u_n-\tilde T(u_n)\|^2(\|u_n\|+\|\tilde T(u_n)\|)^{p-2}\le \frac1a\|I'(u_n)\|_*\|u_n-\tilde T(u_n)\|.$$
Hence, being $(u_n)$ bounded and $\tilde T$ compact (see Proposition \ref{Ttildecompact1}), we have 
$$\|u_n-\tilde T(u_n)\|\le \frac1a \|I'(u_n)\|_* (\|u_n\|+\|\tilde T(u_n)\|)^{2-p}\to0.$$
We conclude that $u_n\to \tilde T(u)=u$ in $W^{1,p}(B)$.
\end{proof}

Lemma \ref{cononelcono} holds for all $1<p<\infty$, hence also in this case the operator $\tilde T$ preserves the cone $\mathcal C_*$ defined in \eqref{Cstar}. 

\begin{lemma}[\textbf{Locally Lipschitz vector field}]\label{KTI1} Let $W:=W^{1,p}(B)\setminus\{u\,:\,\tilde T(u)=u\}$. There exists a locally Lipschitz continuous operator $K: W\to W^{1,p}(B)$ satisfying the following properties: 
\begin{itemize}
\item[(i)] $K(\mathcal C_*\cap W)\subset \mathcal C_*$;
\item[(ii)] $\frac12\|u-K(u)\|\le\|u-\tilde T(u)\|\le2\|u-K(u)\|$ for all $u\in W$;
\item[(iii)] let $a>0$ be the constant given in Proposition \ref{Ttildecompact1}, then 
$$I'(u)[u-K(u)]\ge\frac{a}2 \|u-\tilde T(u)\|^2(\|u\|+\|\tilde T(u)\|)^{p-2}\quad\mbox{for all }u\in W.$$
\end{itemize}
\end{lemma}
\begin{proof} 
By \eqref{J1} it is possible to proceed as in \cite[Lemma 2.1]{BartschLiuWeth}, with $\mathcal D^+:=\mathcal C_*$ and $\mathcal D^-:=\emptyset$.
\end{proof}

\begin{lemma}\label{conseqPS1}
Let $c\in\mathbb R$ be such that $I'(u)\neq 0$ for all $u\in \mathcal C_*$ with $I(u)=c$. Then there exist two positive constants $\bar\varepsilon$ and $\bar\delta$ such that the following inequalities hold  
\begin{itemize}
\item[(i)] $\|I'(u)\|_*\ge\bar\delta$ for all $u\in \mathcal C_*$ with $|I(u)-c|\le 2\bar\varepsilon$;
\item[(ii)] $\|u-K(u)\|\ge\bar\delta$ for all $u\in \mathcal C_*$ with $|I(u)-c|\le 2\bar\varepsilon$.
\end{itemize}
\end{lemma}
\begin{proof} The proof of part (i) is analogous to the one given in Lemma \ref{conseqPS}.
We prove now (ii). Let $$I^{c+2\bar\varepsilon}_{c-2\bar\varepsilon}:= \{u\in \mathcal C_*\,:\,|I(u)-c|\le 2\bar\varepsilon\}.$$ By the part (i),  $I^{c+2\bar\varepsilon}_{c-2\bar\varepsilon}\subset W$, where $W$ is defined in Lemma \ref{KTI1}. Furthermore, for all $u\in I^{c+2\bar\varepsilon}_{c-2\bar\varepsilon}$, $\|u-K(u)\|\ge\frac12\|u-\tilde T(u)\|$ by Lemma \ref{KTI1}-(ii).
Now, by the second inequality of \eqref{J1} and by the (i) part of the present lemma, we have for all $u\in I^{c+2\bar\varepsilon}_{c-2\bar\varepsilon}$ 
$$\|u-\tilde T(u)\|\ge\left(\frac{\|I'(u)\|_*}b\right)^{\frac1{p-1}}\ge\left(\frac{\bar\delta}b\right)^{\frac1{p-1}}.$$
Hence, $\|u-K(u)\|\ge \min\left\{\bar\delta, \frac12\left(\frac{\bar\delta}b\right)^{\frac1{p-1}}\right\}$, still denoted by $\bar\delta$.
\end{proof}

\begin{lemma}\label{lem:eta_bounded}
Let $c\in \R$. The set
\[
\{ \|u\|: \ u\in \mathcal C_* \mbox{ and } I(u)\leq c \}
\]
is bounded by a constant depending only on $c$.
\end{lemma}
\begin{proof}
Let $u\in \mathcal C_*$, then $u\leq u_+$, where $u_+$ is defined in \eqref{u-u+}. Since the function $\tilde{f}$ introduced in Lemma \ref{truncated1} belongs to $\mathfrak{F}$, we have
\[
u_+ = \inf\{ t\in (u_0,\infty): \ g(t)=t^{p-1}\}<\infty
\]
by $(g_4)$. If in addition $I(u)\leq c$, relation \eqref{subcritical1} provides
\[
\frac{\|u\|^p}{p} \leq c+C\int_B(u+u^\ell) dx \leq c+C|B|(u_++u_+^\ell). \qedhere
\]
\end{proof}

\begin{lemma}[\textbf{Descending flow argument}]\label{deformation1} Let $c\in\mathbb R$ be such that $I'(u)\neq 0$ for all $u\in \mathcal C_*$ with $I(u)=c$. Then there exists a function $\eta:\mathcal C_*\to\mathcal C_*$ satisfying the following properties: 
\begin{itemize}
\item[(i)] $\eta$ is continuous with respect to the topology of $W^{1,p}(B)$;
\item[(ii)] $I(\eta(u))\le I(u)$ for all $u\in\mathcal C_*$;
\item[(iii)] $I(\eta(u))\le c-\bar\varepsilon$ for all $u\in\mathcal C_*$ such that $|I(u)-c|<\bar\varepsilon$;
\item[(iv)] $\eta(u)=u$ for all $u\in\mathcal C_*$ such that $|I(u)-c|>2\bar\varepsilon$,
\end{itemize}
where $\bar\varepsilon$ is the positive constant given by Lemma \ref{conseqPS1}.
\end{lemma}
\begin{proof} We define $\eta(t,u)$ as in the first part of the proof of Lemma \ref{deformation}. For all $u\in\mathcal C_*$ and $t>0$ we can write
\begin{equation}\label{eq:flusso_decrescente1}
\begin{aligned}I(\eta(t,u))-I(u)&=\int_0^t\frac{d}{ds}I(\eta(s,u))ds\\
&\hspace{-2.5cm}=-\int_0^t\frac{\chi_1(I(\eta(s,u)))\chi_2(\eta(s,u))}{\|\eta(s,u)-K(\eta(s,u))\|}I'(\eta(s,u))[\eta(s,u)-K(\eta(s,u))]ds\\
&\hspace{-2.5cm}\le-\frac{a}2\displaystyle{\int_0^t\dfrac{\|\eta(s,u)-\tilde T(\eta(s,u))\|^2 \chi_1(I(\eta(s,u)))\chi_2(\eta(s,u))}{\|\eta(s,u)-K(\eta(s,u))\| (\|\eta(s,u)\|+\|\tilde T(\eta(s,u))\|)^{2-p}}}
ds\le0,
\end{aligned}
\end{equation}
where we have used the inequality in Lemma \ref{KTI1}-(iii).

Now, let $u\in\mathcal C_*$ be such that $|I(u)-c|<\bar\varepsilon$ and let $t$ be sufficiently large. Then, two cases arise: either there exists $s\in[0,t]$ for which $I(\eta(s,u))\le c-\bar\varepsilon$ and so, by the previous calculation we get immediately that $I(\eta(t,u))\le c-\bar\varepsilon$, or for all $s\in[0,t]$, $I(\eta(s,u))> c-\bar\varepsilon$. In this second case, 
$$c-\bar\varepsilon< I(\eta(s,u))\le I(u)< c+\bar\varepsilon.$$
In particular, by Lemma \ref{conseqPS1}-(i), $\eta(s,u)\in W$. By the definitions of $\chi_1$ and $\chi_2$, and by Lemma \ref{conseqPS1}-(ii), it results that for all $s\in[0,t]$ 
$$\chi_1(I(\eta(s,u)))=1,\quad\|\eta(s,u)-K(\eta(s,u))\|\ge\bar\delta,\quad\mbox{and}\quad\chi_2(\eta(s,u))=1.$$
Moreover, being $\eta(s,u)\in \mathcal{C}_*$ for every $s\in[0,t]$, Lemmas \ref{lem:eta_bounded} and \ref{Ttildecompact1} provide the existence of a constant $\tilde C$ such $\|\eta(s,u)\|+\|\tilde T(\eta(s,u))\|\le \tilde C$.
Hence, by \eqref{eq:flusso_decrescente1} and Lemma \ref{KTI1}-(ii)-(iii), we obtain
$$I(\eta(t,u))\le I(u)-\frac{a\bar\delta t }{8\tilde C^{2-p}}, 
$$
so that $I(\eta(t,u))\le c-\bar\varepsilon$ for $$t\ge\frac{16\bar\varepsilon \tilde C^{2-p}}{a\bar\delta}.$$

Finally, if we define with abuse of notation $$\eta(u):=\eta\left(\frac{16\bar\varepsilon \tilde C^{2-p}}{a\bar\delta},u\right),$$ 
we have proved that $\eta$ satisfies (ii) and (iii). Properties (i) and (iv) are immediate. The fact that $\eta$ preserves the cone can be proved as in Lemma \ref{deformation}, since $\tilde T(\mathcal C_*)\subset \mathcal C_*$ also for $1<p<2$.
\end{proof}

\begin{proof}[$\bullet$ Proof of Proposition \ref{mountainpass1}] The preliminary results shown in this appendix allow us to prove Proposition \ref{mountainpass1} by proceeding as in the proof of Proposition \ref{mountainpass}.  
\end{proof}

In the case $1<p<2$, we cannot conclude that the mountain pass solution found in Proposition \ref{mountainpass1} is nonconstant.
In particular, Proposition \ref{propp<2} below implies that Lemma \ref{4.9} does not hold for $1<p<2$ and $g(u)=u^{q-1}$. 

Since we are in the pure power case, we refer to the truncated nonlinearity $\tilde f$ defined in \eqref{truncatedpurepower}, namely 
$$
\tilde{f}(s):=\begin{cases}s^{q-1}\quad&\mbox{if }s\in[0,s_0],\\
s_0^{q-1}+\frac{q-1}{\ell-1}s_0^{q-\ell}(s^{\ell-1}-s_0^{\ell-1}) &\mbox{otherwise}.\end{cases}
$$
for some fixed $s_0>\max\{K_{\infty},M\}$ and $\ell\in(p,p^*)$. We introduce the Nehari manifold
\begin{equation}\label{eq:nehari_def2}
\mathcal N:=\{u\in W^{1,p}(B)\setminus\{0\}\,:\, I'(u)[u]=0\}.
\end{equation}

\begin{proposition}\label{propp<2}
Let $1<p<2$ and $g(u)=u^{q-1}$, $q>p$. For every nonconstant $v\in W^{1,p}(B)\cap L^\infty(B)$ there exists $\varepsilon_1(v)>0$ such that the following properties hold for any $0<s<\varepsilon_1(v)$:
\begin{itemize}
\item[(i)] there exists a unique $h(s)>0$ such that $h(s)(1+sv) \in \mathcal{N}$;
\item[(ii)] $I(h(s)(1+sv))-I(1)>0$.
\end{itemize}
\end{proposition} 
\begin{proof} Let $v\in W^{1,p}(B)\cap L^\infty(B)$ be nonconstant. Since $\tilde f\equiv g$ in $[0,s_0]$, from the definition of $\mathcal N$ we can compute explicitly 
\[
h(s)=\left( \frac{\int_B(s^p|\nabla v|^p+|1+sv|^p)dx}{\int_B |1+sv|^qdx} \right)^\frac{1}{q-p}\quad\mbox{for every }0<s<\varepsilon_1(v):=\frac{s_0-1}{\|v\|_{L^\infty(B)}}.
\]
Hence we see that $h(s)$ is unique and regular in $s$. Therefore the proof of (i) is concluded. 

In order to prove (ii), we write the Taylor expansion at the first order of $h$
\[
h(s)=1+h'(0)s+o(s).
\]
By explicit calculations, we arrive at
\[
I(h(s)(1+sv))-I(1)= \frac{s^p}{p}\int_B |\nabla v|^p dx +o(s).
\]
Since $v$ is nonconstant, the statement follows.
\end{proof}

\section*{Acknowledgments}
\noindent F. Colasuonno was partially supported by the INdAM - GNAMPA Project 2016 ``Fenomeni non-locali: teoria, metodi e applicazioni". B. Noris was partially supported by the project ERC Advanced Grant 2013 n. 339958: ``Complex Patterns for Strongly Interacting Dynamical Systems --
COMPAT''. 

\bibliographystyle{abbrv}
\bibliography{biblio}

\end{document}